\documentclass[openright,11pt]{article}

\usepackage[latin1]{inputenc}
\usepackage[english]{babel}
\usepackage{here,enumerate,graphicx}
\usepackage{amsmath,amssymb,amscd,amsthm,mathrsfs}

\usepackage[flushmargin]{footmisc}
\newcommand\blfootnote[1]{%
	\begingroup
	\renewcommand\thefootnote{}\footnote{#1}%
	\addtocounter{footnote}{-1}%
	\endgroup}

\theoremstyle{plain}
\newtheorem{theorem}{Theorem}[section]
\newtheorem{proposition}[theorem]{Proposition}

\newtheorem{lemma}[theorem]{Lemma}
\newtheorem{corollary}[theorem]{Corollary}

\newtheorem{fact}{Fact}[subsection]

\theoremstyle{definition}
\newtheorem{definition}[theorem]{Definition}
\newtheorem{remark}[theorem]{Remark}

\newtheorem{question}[theorem]{Question}

\newcommand{\NN}{\mathbb{N}}
\newcommand{\RR}{\mathbb{R}}
\newcommand{\CC}{\mathbb{C}}
\newcommand{\TT}{\mathbb{T}}

\newcommand{\Bi}{\mathscr{B}}
\newcommand{\Ci}{\mathscr{C}}
\newcommand{\Del}{\Delta}
\newcommand{\Ec}{\mathcal{E}}
\newcommand{\Fc}{\mathcal{F}}
\newcommand{\IP}{\mathcal{IP}}
\newcommand{\Kc}{\mathcal{K}}
\newcommand{\Lc}{\mathcal{L}}
\newcommand{\mf}{\mathfrak{m}}
\newcommand{\Part}{\mathscr{P}}

\newcommand{\Sc}{\mathcal{S}}
\newcommand{\Uc}{\mathcal{U}}

\newcommand{\Per}{\operatorname{Per}}
\newcommand{\URec}{\operatorname{URec}}
\newcommand{\FRec}{\operatorname{FRec}}
\newcommand{\UFRec}{\operatorname{UFRec}}
\newcommand{\RRec}{\operatorname{RRec}}
\newcommand{\Rec}{\textup{Rec}}

\newcommand{\FHC}{\operatorname{FHC}}
\newcommand{\UFHC}{\operatorname{UFHC}}
\newcommand{\RHC}{\operatorname{RHC}}
\newcommand{\HC}{\textup{HC}}

\newcommand{\cl}{\overline}
\newcommand{\orb}{\operatorname{Orb}}
\newcommand{\lspan}{\operatorname{span}}
\newcommand{\res}{\arrowvert}
\newcommand{\eps}{\varepsilon}
\newcommand{\card}{\operatorname{card}}

\newcommand{\QEDh}{\pushQED{\qed}\qedhere} 

\newcommand{\Bdsup}{\overline{\operatorname{Bd}}}
\newcommand{\dsup}{\overline{\operatorname{dens}}}
\newcommand{\dinf}{\underline{\operatorname{dens}}}
\newcommand{\dens}{\operatorname{dens}}

\def\1{{\mathchoice {\rm 1\mskip-4mu l} {\rm 1\mskip-4mu l} {\rm 1\mskip-4.5mu l} {\rm 1\mskip-5mu l}}}

\newcommand{\supp}{\operatorname{supp}}

\newcommand{\symbf}{\boldsymbol}

\newcommand{\ep}[2]{\langle #1 , #2 \rangle}

\begin{document}
\begin{center}
	\begin{LARGE}
		{\bf Recurrence properties for linear dynamical systems: An approach via invariant measures}
	\end{LARGE}
\end{center}

\begin{center}
	\begin{Large}
		by
	\end{Large}
\end{center}

\begin{center}
	\begin{Large}
		Sophie Grivaux \& Antoni L\'opez-Mart\'inez\blfootnote{\textbf{Acknowledgments}: The first author was supported by the project FRONT of the French National Research Agency (grant ANR-17-CE40-0021) and by the Labex CEMPI (ANR-11-LABX-0007-01). The second author was partially supported by the Spanish Ministerio de Ciencia, Innovaci\'on y Universidades, grant FPU2019/04094, and partially by MCIN/AEI/10.13039/501100011033, Project PID2019-105011GB-I00.Unimodular eigenvectors, Uniformly recurrent operators.\\ \textbf{Journal-ref}: Journal de Math\'ematiques Pures et Appliqu\'ees, Volume 169, January 2023, Pages 155-188, https://doi.org/10.1016/j.matpur.2022.11.011}
	\end{Large}
\end{center}


\begin{abstract}
	 We study different pointwise recurrence notions for linear dynamical systems from the Ergodic Theory point of view. We show that from any reiteratively recurrent vector $x_0$, for an adjoint operator $T$ on a separable dual Banach space $X$, one can construct a $T$-invariant probability measure which contains $x_0$ in its support. This allows us to establish some equivalences, for these operators, between some strong pointwise recurrence notions which in general are completely distinguished. In particular, we show that (in our framework) reiterative recurrence coincides with frequent recurrence; for complex Hilbert spaces uniform recurrence coincides with the property of having a spanning family of unimodular eigenvectors; and the same happens for power-bounded operators on complex reflexive Banach spaces. These (surprising) properties are easily generalized to product and inverse dynamical systems, which implies some relations with the respective hypercyclicity notions. Finally we study how typical is an operator with a non-zero reiteratively recurrent vector in the sense of Baire category.
\end{abstract}


\section{Introduction and main results}

\subsection{General background}\label{Subsec:1.1GB}

This paper focuses on some aspects of the interplay between topological and measurable dynamics for {\em linear dynamical systems}, our main aim being to investigate the links in this context between various notions of {\em recurrence}.\\[-5pt]

A ({\em real} or {\em complex}) {\em linear dynamical system} $(X,T)$ is given by the action of a bounded linear operator $T$ on a (real or complex) separable infinite-dimensional Banach space $X$, and we will denote by $\Lc(X)$ the {\em set of bounded linear operators} acting on such a space $X$. A linear dynamical system is a particular case of a {\em Polish dynamical system} (i.e. a system given by the action of a continuous map on a separable completely metrizable space), and some of the results obtained in the paper hold in this more general context. Given a dynamical system $T:X\rightarrow X$ and a point $x \in X$ we will denote by
\[
\orb(x,T) := \{ T^nx : n \in \NN_0 \},
\]
the {\em $T$-orbit of $x$}, where $\NN_0:=\NN\cup\{0\}$. Examples of topological properties, for linear dynamical systems, which will be of interest to us in this work are:
\begin{enumerate}[(a)]
	\item \textit{recurrence}: the operator $T$ is said to be {\em recurrent} if the set
	\[
	\Rec(T) := \left\{ x \in X : x \in \cl{\orb(Tx,T)} \right\},
	\]
	is dense in $X$, where each vector $x \in \Rec(T)$ is called a {\em recurrent vector} for $T$. By the (not so well-known) Costakis-Manoussos-Parissis' theorem (see \cite[Proposition 2.1]{CoMaPa2014}), this notion is equivalent to that of {\em topological recurrence}, i.e. for each non-empty open subset $U$ of $X$ one can find $n \in \NN$ such that $T^n(U) \cap U \neq \varnothing$; and in this case, the set $\Rec(T)$ of recurrent vectors for $T$ is a dense $G_{\delta}$ subset of $X$;
	
	\item \textit{hypercyclicity}: the operator $T$ is said to be {\em hypercyclic} if there exists a vector $x \in X$, called a {\em hypercyclic vector} for $T$, whose orbit $\orb(x,T)$ is dense in $X$. By the Birkhoff's Transitivity theorem (see \cite[Theorem 1.16]{GrPe2011}), this notion is equivalent to that of {\em topological transitivity}, i.e. for each pair $U,V$ of non-empty open subsets of $X$ one can find $n \in \NN_0$ such that $T^n(U) \cap V \neq \varnothing$; and in this case, the set of hypercyclic vectors for $T$, denoted by $\HC(T)$, is a dense $G_{\delta}$ subset of $X$.
\end{enumerate}
If given a point $x \in X$ and a set $A \subset X$ we denote the {\em return set from $x$ to $A$} as
\[
N_T(x,A) := \{ n\in\NN_0 : T^nx \in A \},
\]
which will be denoted by $N(x,A)$ if no confusion with the map arises, we can reformulate the above notions in the following terms: a vector $x \in X$ is recurrent if and only if $N(x,U)$ is an infinite set for every neighbourhood $U$ of $x$; and a vector $x \in X$ is hypercyclic if and only if $N(x,U)$ is an infinite set for every non-empty open subset $U$ of $X$. Historically, hypercyclicity and its generalizations have been the most studied notions in linear dynamics while the systematic study of the linear dynamical recurrence-kind properties started recently in 2014 with \cite{CoMaPa2014}, in spite of the great non-linear dynamical knowledge already existing in this area (see for instance \cite{Furstenberg1981}).\\[-5pt]

Direct relations between these properties and Ergodic Theory arise when we are able to consider a {\em probability} (or a {\em positive finite}) {\em Borel measure} $\mu$ on $X$ (i.e. defined on $\Bi(X)$, the {\em $\sigma$-algebra of Borel sets} of $X$), which will sometimes be required to have {\em full support} (i.e. $\mu(U)>0$ for every non-empty open subset $U$ of $X$). We will only consider Borel measures in this work, and the word ``Borel'' will sometimes be omitted. If such a measure $\mu$ exists, we can study the dynamical system $(X,\Bi(X),\mu,T)$ from the point of view of Ergodic Theory and relevant properties are:
\begin{enumerate}[(a)]
	\item \textit{invariance}: the operator $T$ is said to be {\em $\mu$-invariant}, or equivalently, the measure $\mu$ is called {\em $T$-invariant}, if for each $A \in \Bi(X)$ the equality $\mu(T^{-1}(A))=\mu(A)$ holds. By the Poincar\'e's Recurrence theorem (see \cite[Theorem 1.4]{Walters1982}), this notion implies that for every $A \in \Bi(X)$ with $\mu(A)>0$ there is $n \in \NN$ such that $T^n(A) \cap A \neq \varnothing$. The {\em Dirac mass} $\delta_0$ at $0$ is always an invariant measure for any operator $T$, and we will say that a $T$-invariant probability measure $\mu$ is {\em non-trivial} if it is different from $\delta_0$.
	
	\item \textit{ergodicity}: the operator $T$ is said to be {\em ergodic} with respect to $\mu$, provided that the measure $\mu$ is $T$-invariant, and for each $A \in \Bi(X)$ with $T^{-1}(A) = A$ we have that $\mu(A) \in \{0,1\}$. It is well known that the last statement is equivalent to the fact that, for each pair of sets $A,B \in \Bi(X)$ with $\mu(A),\mu(B)>0$ there is $n \in \NN_0$ such that $\mu\big(T^{-n}(A) \cap B\big)>0$ (see~\cite[Theorem~1.5]{Walters1982}).
\end{enumerate}
When $T$ is ergodic with respect to a measure with full support, it follows from the Birkhoff's Pointwise Ergodic theorem that $T$ is not only hypercyclic, but even {\em frequently hypercyclic}: there exists a vector $x \in X$ such that for each non-empty open subset $U$ of $X$ the return set $N(x,U)$ has {\em positive lower density}; in other words:
\[
\dinf(N(x,U)) = \liminf_{N \to \infty} \dfrac{\card(N(x,U)\cap [0,N])}{N+1} > 0.
\]
Such a vector $x$ is said to be a {\em frequently hypercyclic vector} for $T$, and the set of all frequently hypercyclic vectors is denoted by $\FHC(T)$. See \cite[Corollary 5.5]{BaMa2009} for the details of this argument, and for more on frequent hypercyclicity.\\[-5pt]

When $T$ is only supposed to admit an invariant measure $\mu$, it follows easily from the Poincar\'e's Recurrence theorem that $\mu$-almost every $x \in X$ is a recurrent point for $T$ (see~\cite[Theorem~3.3]{Furstenberg1981}). Our main line of thought in this work will be to connect various (stronger) notions of recurrence via invariant measures, proceeding essentially in two steps:
\begin{enumerate}[-]
	\item if $T$ admits vectors with a certain (rather weak) recurrence property, prove that it admits a non-trivial invariant measure, perhaps with full support (see Theorem~\ref{The:Main});
	
	\item if $T$ admits a non-trivial invariant measure (perhaps with full support), prove that it admits vectors with a certain strong recurrence property (see Lemmas~\ref{Lem:Sophie} and \ref{Lem:Eigenvectors}).
\end{enumerate}
This approach in the context of linear dynamical systems comes from the paper \cite{GriMa2014}, which extends to the linear setting some well-known results in the context of {\em compact dynamical systems} (see \cite[Chapter 3 and Lemma 3.17]{Furstenberg1981}). The various recurrence notions which we will consider were introduced and studied in the work \cite{BoGrLoPe2020}, but the initial study of recurrence in linear dynamics started in \cite{CoMaPa2014}. In the next subsection, we recall the relevant definitions and present the first main result of this paper.

\subsection{Furstenberg families: recurrence and hypercyclicity notions}\label{Subsec:1.2FFRH}

The Banach spaces $X$ considered in this subsection can be either real or complex. Let us first recall the following definitions from \cite{BoGrLoPe2020}:

\begin{definition}
	Given a non-empty collection of sets $\Fc \subset \mathscr{P}(\NN_0)$ we say that it is a {\em Furstenberg family} if for each $A \in \Fc$ we have
	\begin{enumerate}[(i)]
		\item $A$ is infinite;
		
		\item if $A \subset B \subset \NN_0$ then $B \in \Fc$.
	\end{enumerate}
	The {\em dual family} of $\Fc$ is defined as the collection of sets
	\[
	\Fc^* := \{ A \subset \NN_0 \text{ infinite} : A \cap B\neq\varnothing \text{ for all } B \in \Fc \}.
	\]
\end{definition}

\begin{definition}\label{Def:F-rec-hyp}
	Let $(X,T)$ be a linear dynamical system and let $\Fc$ be a Furstenberg family. A point $x \in X$ is said to be {\em $\Fc$-recurrent} (resp. {\em $\Fc$-hypercyclic}) if $N(x,U) \in \Fc$ for every neighbourhood $U$ of $x$ (resp. for every non-empty open subset $U$ of $X$). We will denote by $\Fc\Rec(T)$ (resp. $\Fc\HC(T)$) the set of such points and we say that $T$ is {\em $\Fc$-recurrent} (resp. {\em $\Fc$-hypercyclic}) if $\Fc\Rec(T)$ is dense in $X$ (resp. if $\Fc\HC(T) \neq \varnothing$).
\end{definition}

The families $\Fc$ for which there exist $\Fc$-hypercyclic operators are by far less common than those for which $\Fc$-recurrence exists since having an orbit distributed around the whole space is much more complicated than having it just around the initial point of the orbit. Furstenberg families associated just to recurrence will be used in the following subsection, but in the present one we focus on the most known cases of families for which both notions exist. In particular, a point $x \in X$ is said to be 
\begin{enumerate}[(a)]
	\item {\em frequently recurrent} (resp. {\em frequently hypercyclic}) if $\dinf(N(x,U))>0$ for every neighbourhood $U$ of $x$ (resp. for every non-empty open subset $U$ of $X$). We will denote by $\FRec(T)$ (resp. $\FHC(T)$) the set of such points, and we say that $T$ is {\em frequently recurrent} (resp. {\em frequently hypercyclic}) if $\FRec(T)$ is dense in $X$ (resp. if $\FHC(T) \neq \varnothing$);
	
	\item {\em $\Uc$-frequently recurrent} (resp. {\em $\Uc$-frequently hypercyclic}) if $\dsup(N(x,U))>0$ for every neighbourhood $U$ of $x$ (resp. for every non-empty open subset $U$ of $X$). We will denote by $\UFRec(T)$ (resp. $\UFHC(T)$) the set of such points, and we say that $T$ is {\em $\Uc$-frequently recurrent} (resp. {\em $\Uc$-frequently hypercyclic}) if $\UFRec(T)$ is dense in $X$ (resp. if $\UFHC(T) \neq \varnothing$);
	
	\item {\em reiteratively recurrent} (resp. {\em reiteratively hypercyclic}) if $\Bdsup(N(x,U))>0$ for every neighbourhood $U$ of $x$ (resp. for every non-empty open subset $U$ of $X$). We will denote by $\RRec(T)$ (resp. $\RHC(T)$) the set of such points and we say that $T$ is {\em reiteratively recurrent} (resp. {\em reiteratively hypercyclic}) if $\RRec(T)$ is dense in $X$ (resp. if $\RHC(T) \neq \varnothing$);
\end{enumerate}
where for any $A \subset \NN_0$ its:
\begin{enumerate}[(a)]
	\item {\em lower density} is $\displaystyle\dinf(A) := \liminf_{N \to \infty} \dfrac{\card(A\cap [0,N])}{N+1}$;
	
	\item {\em upper density} is $\displaystyle\dsup(A) := \limsup_{N \to \infty} \dfrac{\card(A\cap [0,N])}{N+1}$;
	
	\item {\em upper Banach density} is $\displaystyle\Bdsup(A) := \limsup_{N \to \infty} \left( \max_{m \geq 0} \dfrac{\card(A\cap [m,m+N])}{N+1} \right)$.
\end{enumerate}
The introduced notions follow Definition \ref{Def:F-rec-hyp} applied to the respective families of {\em positive} ({\em lower}, {\em upper} and {\em upper Banach}) {\em density sets}, and in fact, the inequalities between the respective densities imply the inclusions $\FRec(T) \subset \UFRec(T) \subset \RRec(T) \subset \Rec(T)$ and $\FHC(T) \subset \UFHC(T) \subset \RHC(T) \subset \HC(T)$. In particular, frequent, $\Uc$-frequent and reiterative recurrence are clearly stronger notions than ``usual'' recurrence as defined in Subsection \ref{Subsec:1.1GB}, and frequent recurrence is a stronger notion than $\Uc$-frequent recurrence, which is in its turn stronger than reiterative recurrence.\\[-5pt]

We point out that all these notions are not specific to the linear setting; we will actually use them in the context of Polish dynamical systems in Sections \ref{Sec:2IMfromRRec}, \ref{Sec:3FRecfromRRec}, \ref{Sec:5Product} and \ref{Sec:6Inverse}. However, since we are focused on linear dynamical systems, our first main result connects all of them in the framework of {\em adjoint operators on separable dual Banach spaces}:

\begin{theorem}\label{The:Banach}
	Let $T:X\rightarrow X$ be an adjoint operator on a (real or complex) separable dual Banach space $X$. Then we have the equality
	\[
	\cl{\FRec(T)} = \cl{\RRec(T)}.
	\]\newpage
	Moreover:
	\begin{enumerate}[{\em(a)}]
		\item The following statements are equivalent:
		\begin{enumerate}[{\em(i)}]
			\item $\FRec(T) \setminus\{0\} \neq \varnothing$;
			
			\item $\UFRec(T) \setminus\{0\} \neq \varnothing$;
			
			\item $\RRec(T) \setminus\{0\} \neq \varnothing$;
			
			\item $T$ admits a non-trivial invariant probability measure.
		\end{enumerate}
		
		\item The following statements are equivalent:
		\begin{enumerate}[{\em(i)}]
			\item $T$ is frequently recurrent;
			
			\item $T$ is $\Uc$-frequently recurrent;
			
			\item $T$ is reiteratively recurrent;
			
			\item $T$ admits an invariant probability measure with full support.
		\end{enumerate}
	\end{enumerate}
	In particular, these results hold whenever $T$ is an operator on a (real or complex) separable reflexive Banach space $X$.
\end{theorem}

The above theorem is in spirit similar to \cite[Theorem~1.3]{GriMa2014}, where it is proved that every ($\Uc$-)frequently hypercyclic operator on a separable reflexive space admits an invariant measure with full support. It is observed in \cite[Remark~3.5]{GriMa2014} that the arguments extend to every adjoint operator acting on a separable dual Banach space. It is also proved in \cite[Proposition~2.11]{GriMa2014} that, in this same setting, operators admitting an invariant measure with full support are exactly those which are frequently recurrent. However, the notion of frequent recurrence introduced in \cite[Section~2.5]{GriMa2014} is rather different from the one given in Definition~\ref{Def:F-rec-hyp} since in \cite{GriMa2014}, $T \in \Lc(X)$ is called frequently recurrent if for every non-empty open subset $U$ of $X$ there exists a vector $x_U \in U$ for which just the positive lower density of the return set $N(x_U,U)$ is required. This notion is (at least formally) weaker than the one used here (see Remark~\ref{Rem:GriMa2014RRec}).\\[-5pt]

The proof of Theorem \ref{The:Banach} relies on some modifications of the arguments of \cite[Section~2]{GriMa2014}, which will be presented in Sections \ref{Sec:2IMfromRRec} and \ref{Sec:3FRecfromRRec}. We mention that it cannot be extended to all operators acting on separable (infinite-dimensional) Banach spaces. Indeed, it is shown in \cite[Theorem~5.7 and Corollary~5.8]{BoGrLoPe2020} that there even exist reiteratively hypercyclic operators on the space $c_0(\NN)$ which do not admit any non-zero $\Uc$-frequently recurrent vector.

\subsection{Uniform, $\IP^*$, $\Del^*$-recurrence and unimodular eigenvectors}\label{Subsec:1.3UIP*D*UEig}

In this subsection the underlying Banach spaces $X$ are assumed to be complex. A vector $x \in X$ is a {\em unimodular eigenvector} for $T$ provided $x \neq 0$ and $Tx=\lambda x$ for some unimodular complex number $\lambda \in \TT = \{ z \in \CC : |z|=1 \}$. The set of unimodular eigenvectors of $T$ will be denoted by $\Ec(T)$, and
\[
\Ec(T) = \bigcup_{\lambda \in \TT} \ker(\lambda-T) \setminus\{0\}.
\]
Unimodular eigenvectors are clearly frequently recurrent vectors for $T$, but they enjoy some stronger recurrence properties like uniform, $\IP^*$ and even $\Del^*$-recurrence (see Definition~\ref{Def:unif_IP*_Del*-rec} below). Our general aim in this paper is to investigate some contexts in which these strong forms of recurrence actually imply the existence of unimodular eigenvectors. We will see that it is indeed the case in (at least) the following two situations:
\begin{enumerate}[-]
	\item when $T$ is an operator on a complex Hilbert space (see Theorem~\ref{The:Hilbert} below);
	
	\item when $T$ is a power-bounded operator on a complex reflexive Banach space (Theorem~\ref{The:JaDeGli+URec}).
\end{enumerate}

Let us now introduce these stronger recurrence notions which are defined by considering Furstenberg families only relevant for the notion of recurrence, and, contrary to those used in Subsection~\ref{Subsec:1.2FFRH}, having no hypercyclicity analogue.

\begin{definition}
	Let $A \subset \NN_0$. We say that $A$ is a
	\begin{enumerate}[(a)]		
		\item {\em syndetic} set, if there is $m \in \NN$ such that for every $x \in \NN_0$ we have $[x,x+m]\cap A\neq\varnothing$. We will denote by
		\[
		\Sc := \{ A \subset \NN_0 : A \text{ is syndetic} \},
		\]
		the {\em Furstenberg family of syndetic sets}.
		
		\item {\em IP-set}, if there is an increasing sequence $(x_n)_{n=1}^{\infty} \in \NN_0^{\NN}$ such that
		\[
		\left\{ \sum_{n \in F} x_n : F \subset \NN \text{ finite} \right\} \subset A.
		\]
		We will denote by
		\[
		\IP := \{ A \subset \NN_0 : A \text{ is an IP-set} \},
		\]
		the {\em Furstenberg family of IP-sets}.
		
		\item {\em $\Del$-set}, if there is an infinite set $B \subset \NN_0$ such that $(B-B)\cap\NN\subset A$. We will denote by
		\[
		\Del := \{ A \subset \NN_0 : A \text{ is a $\Del$-set} \},
		\]
		the {\em Furstenberg family of $\Delta$-sets}.
	\end{enumerate}
\end{definition}

From Definition \ref{Def:F-rec-hyp} and the dual families notation we have:

\begin{definition}\label{Def:unif_IP*_Del*-rec}
	Let $(X,T)$ be a linear dynamical system. A point $x \in X$ is said to be 
	\begin{enumerate}[(a)]
		\item {\em uniformly recurrent} if $N(x,U) \in \Sc$ for every neighbourhood $U$ of $x$. We will denote by $\URec(T)$ the set of such points and $T$ is {\em uniformly recurrent} if $\URec(T)$ is dense in $X$;
		
		\item {\em $\IP^*$-recurrent} if $N(x,U) \in \IP^*$ for every neighbourhood $U$ of $x$. We will denote by $\IP^*\Rec(T)$ the set of such points, and $T$ is {\em $\IP^*$-recurrent} if $\IP^*\Rec(T)$ is dense in $X$;
		
		\item {\em $\Del^*$-recurrent} if $N(x,U) \in \Del^*$ for every neighbourhood $U$ of $x$. We will denote by $\Del^*\Rec(T)$ the set of such points, and $T$ is {\em $\Del^*$-recurrent} if $\Del^*\Rec(T)$ is dense in $X$.
	\end{enumerate}
\end{definition}
It is shown in \cite[Proposition 2]{BMPP2016} that the above Furstenberg families do not admit a respective hypercyclicity counterpart. As in the previous subsection these recurrence notions could be defined for (non-linear) Polish dynamical systems, but since the eigenvectors will play a fundamental role in the connection between those concepts we will directly work with complex linear maps. The relation $\Del^* \subset \IP^* \subset \Sc$ between the families (see \cite{BerDown2008}), Proposition~\ref{Pro:unimodular} and \cite{BoGrLoPe2020} imply the inclusions
\[
\lspan(\Ec(T)) \subset \Del^*\Rec(T) \subset \IP^*\Rec(T) \subset \URec(T) \subset \FRec(T) \subset \UFRec(T) \subset \RRec(T).
\]
From there the following question was proposed in \cite{BoGrLoPe2020}:

\begin{question}[\textbf{\cite[Question 6.3]{BoGrLoPe2020}}]\label{Q:urecIP*rec}
	Does there exist an operator which is uniformly recurrent but not $\IP^*$-recurrent?
\end{question}

The uniformly recurrent operators considered in \cite{BoGrLoPe2020} were also $\IP^*$-recurrent, and in fact a partial negative answer to Question \ref{Q:urecIP*rec} was already given in \cite[Theorem 6.2]{BoGrLoPe2020} for the particular case of power-bounded operators, condition which implies the equality of the two sets $\IP^*\Rec(T)$ and $\URec(T)$. The second main result of this paper provides a negative answer to Question \ref{Q:urecIP*rec} for operators acting on a complex separable Hilbert space $H$, by showing the following stronger statement: {\em any uniformly recurrent operator $T \in \Lc(H)$ has a spanning set of unimodular eigenvectors}. More precisely, define the sets
\begin{eqnarray}
	\FRec^{bo}(T) &:=& \FRec(T) \cap \{ x \in H \textup{ with bounded $T$-orbit} \}; \nonumber\\[5pt]
	\UFRec^{bo}(T) &:=& \UFRec(T) \cap \{ x \in H \textup{ with bounded $T$-orbit} \}; \nonumber\\[5pt]
	\RRec^{bo}(T) &:=& \RRec(T) \cap \{ x \in H \textup{ with bounded $T$-orbit} \}. \nonumber
\end{eqnarray}
Since uniformly recurrent vectors have bounded orbit, we have $\URec(T) \subset \FRec^{bo}(T) \subset \UFRec^{bo}(T) \subset \RRec^{bo}(T)$ and hence:

\begin{theorem}\label{The:Hilbert}
	Let $T \in \Lc(H)$ where $H$ is a complex separable Hilbert space. Then we have the equalities
	\[
	\cl{\lspan(\Ec(T))} = \cl{\URec(T)} = \cl{\RRec^{bo}(T)}.
	\]
	Moreover:
	\begin{enumerate}[{\em(a)}]
		\item The following statements are equivalent:
		\begin{enumerate}[{\em(i)}]			
			\item $\Ec(T)\neq\varnothing$;
			
			\item $\Del^*\Rec(T) \setminus \{0\} \neq \varnothing$;
			
			\item $\IP^*\Rec(T) \setminus \{0\} \neq \varnothing$;
			
			\item $\URec(T) \setminus \{0\} \neq \varnothing$;
			
			\item $\FRec^{bo}(T) \setminus\{0\} \neq \varnothing$;
			
			\item $\UFRec^{bo}(T) \setminus\{0\} \neq \varnothing$;
			
			\item $\RRec^{bo}(T) \setminus\{0\} \neq \varnothing$;
			
			\item $T$ admits a non-trivial invariant probability measure $\mu$ with $\int_H \|z\|^2 d\mu(z) < \infty$.
		\end{enumerate}
		
		\item The following statements are equivalent:
		\begin{enumerate}[{\em(i)}]
			\item the set $\lspan(\Ec(T))$ is dense in $H$;
			
			\item $T$ is $\Del^*$-recurrent;
			
			\item $T$ is $\IP^*$-recurrent;
			
			\item $T$ is uniformly recurrent;
			
			\item the set $\FRec^{bo}(T)$ is dense in $H$;
			
			\item the set $\UFRec^{bo}(T)$ is dense in $H$;
			
			\item the set $\RRec^{bo}(T)$ is dense in $H$;
			
			\item $T$ admits an invariant probability measure $\mu$ with full support and $\int_H \|z\|^2 d\mu(z)~<~\infty$.
		\end{enumerate}
	\end{enumerate}
\end{theorem}

The proof of Theorem \ref{The:Hilbert} is really specific to the Hilbertian setting in a somewhat roundabout way. It relies on the following three main arguments:
\begin{enumerate}[-]
	\item the existence of a non-trivial invariant measure with a {\em finite second-order moment}, under the assumption of the existence of a reiteratively recurrent vector with {\em bounded orbit}; this argument is the same as the one employed in the proof of Theorem \ref{The:Banach} above;
	
	\item the fact that any operator on a space of type 2, admitting an invariant measure with a finite second-order moment, admits in fact a Gaussian invariant measure whose support contains that of the initial measure (see Remark \ref{Rem:co-type2BlackBox});
	
	\item and lastly, the fact that on spaces of cotype 2, the existence of a Gaussian invariant measure for an operator $T$ implies that the unimodular eigenvectors of $T$ span a dense subspace of the support of the measure (see {\em Step} 3 of Lemma \ref{Lem:Eigenvectors}).
\end{enumerate}
These last two ``facts'' are far from being trivial, and we refer the reader to \cite[Chapter~5]{BaMa2009} for a proof, as well as for an introduction to the role of Gaussian measures in linear dynamics. Since the only spaces which are both of type 2 and of cotype 2 are those which are isomorphic to a Hilbert space, our proof of Theorem \ref{The:Hilbert} does not seem to admit any possible extension to a non-Hilbertian setting. The following question remains widely open:

\begin{question}\label{Q:urecEspan}
	Let $X$ be a complex Banach space and let $T:X\rightarrow X$ be a uniformly recurrent operator. Is $\lspan(\Ec(T))$ a dense set in $X$? What about the cases where $T$ is an adjoint operator on a separable dual Banach space or where $X$ is a reflexive Banach space?
\end{question}

A partial (but not completely satisfactory) answer is our third and last main result, which only concerns the {\em power-bounded} operators on complex reflexive Banach spaces. It extends \cite[Theorem~6.2]{BoGrLoPe2020} by showing that such an operator $T \in \Lc(X)$ is again uniformly recurrent if and only if has a spanning set of unimodular eigenvectors. More precisely, we have:

\begin{theorem}\label{The:JaDeGli+URec}
	Let $T:X\rightarrow X$ be a power-bounded operator on a complex reflexive Banach space $X$. Then we have the equality
	\[
	\cl{\lspan(\Ec(T))} = \cl{\URec(T)}.
	\]
	In particular:
	\begin{enumerate}[{\em(a)}]
		\item The following statements are equivalent:
		\begin{enumerate}[{\em(i)}]
			\item $\Ec(T)\neq\varnothing$;
			
			\item $\Del^*\Rec(T) \setminus \{0\} \neq \varnothing$;
			
			\item $\IP^*\Rec(T) \setminus \{0\} \neq \varnothing$;
			
			\item $\URec(T) \setminus \{0\} \neq \varnothing$.
		\end{enumerate}
		
		\item The following statements are equivalent:
		\begin{enumerate}[{\em(i)}]
			\item the set $\lspan(\Ec(T))$ is dense in $X$;
			
			\item $T$ is $\Del^*$-recurrent;
			
			\item $T$ is $\IP^*$-recurrent;
			
			\item $T$ is uniformly recurrent.
		\end{enumerate}
	\end{enumerate}
\end{theorem}

The proof of Theorem \ref{The:JaDeGli+URec} relies on the splitting theorem of Jacobs-Deleeuw-Glicksberg (see \cite[Section 2.4]{Krengel1985}). Here the unimodular eigenvectors are obtained in a very different way than in the proof of Theorem \ref{The:Hilbert} (via characters on a certain compact abelian group).\\[-5pt]

Even though the arguments used in the proofs of the two theorems above still hold for complex finite-dimensional spaces, in this situation one can use directly the canonical Jordan decomposition (see \cite[Theorem 4.1]{CoMaPa2014} and \cite[Theorem 7.3]{BoGrLoPe2020}) to obtain a spanning set of unimodular eigenvectors even from ``usual'' recurrence as defined in Subsection \ref{Subsec:1.1GB}.

\subsection{Organization of the paper}

Section \ref{Sec:2IMfromRRec} is devoted to the statement and proof of a purely non-linear result (Theorem~\ref{The:Main}) which allows to construct invariant measures from reiteratively recurrent points for a rather general class of Polish dynamical systems (which includes the compact ones). Theorem~\ref{The:Main} is a modest improvement of \cite[Theorem 1.5, Remarks 2.6 and 2.12]{GriMa2014} and its proof is based on a modification of the construction given in \cite[Section 2]{GriMa2014}. In Section \ref{Sec:3FRecfromRRec}, we prove some results where frequent recurrence is deduced from reiterative recurrence, in particular Theorem~\ref{The:Banach}. Theorems \ref{The:Hilbert} and \ref{The:JaDeGli+URec}, which provide links between strong forms of recurrence and the existence of unimodular eigenvectors, are proved in Section \ref{Sec:4EigfromURec}. Sections~\ref{Sec:5Product} and \ref{Sec:6Inverse} present some applications of the above results in terms of product and inverse dynamical systems respectively, while we study in Section \ref{Sec:7Typical} the ``typicality'', in the Baire Category sense, of some recurrence properties. Lastly, we gather in Section \ref{Sec:8Open} some possibly interesting open questions and a few comments related to them.

\section{Invariant measures from reiterative recurrence}\label{Sec:2IMfromRRec}

In this section, we present a modification of the construction of \cite[Section 2]{GriMa2014} which allows to construct invariant measures from reiteratively recurrent points for a rather general class of Polish dynamical systems, including the compact ones (see Remark \ref{Rem:compact}).

\subsection{Topological assumptions and initial comments}

We begin this section with some notation: whenever we consider a space of functions we will use the symbol $\1$ to denote the function constantly equal to $1$, and given a subset $A$ of the domain of the functions, we will write $\1_A$ for the indicator function of $A$, i.e. $\1_A(x)=1$ if $x \in A$ and $\1_A(x)=0$ if $x \notin A$. For instance, if we consider $\ell^{\infty}=\ell^{\infty}(\NN)$, the {\em space of all bounded sequences of real numbers}, $\1 \in \ell^{\infty}$ is the sequence with all its terms equal to $1$, and for every $A \subset \NN$, $\1_A \in \ell^{\infty}$ will be the sequence in which the $n$-th coordinate is $1$ if $n \in A$ and $0$ otherwise.\\[-5pt]

A {\em Banach limit} is a continuous functional $\mf:\ell^{\infty}\rightarrow \RR$ such that for every pair of sequences $\phi=(\phi(n))_{n\geq1},\psi=(\psi(n))_{n\geq1} \in \ell^{\infty}$, every $\alpha,\beta \in \RR$ and every $a \in \NN$:
\begin{enumerate}[(a)]
	\item $\mf(\alpha\phi + \beta\psi) = \alpha\mf(\phi) + \beta\mf(\psi)$ (linearity);
	
	\item $\phi(n)\geq 0$ for every $n \in \NN$ implies $\mf(\phi)\geq 0$ (positivity);
	
	\item $\mf((\phi(n+a)_{n\geq1}) = \mf((\phi(n))_{n\geq1})$ (shift-invariance);
	
	\item if $\phi$ is a convergent sequence then $\mf(\phi)=\lim_{n\to\infty} \phi(n)$ (which implies $\mf(\1)=1$).
\end{enumerate}

Following \cite{GriMa2014}, each Banach limit $\mf$ should be viewed as a finitely additive measure on $\NN$. In fact we will write the result of the action of $\mf$ on a ``function'' $\phi \in \ell^{\infty}$ as the integral:
\[
\mf(\phi) = \int_{\NN} \phi(i) d\mf(i).
\]

Given a topological space $(X,\tau)$ we will denote by $\Bi(X,\tau)$ its {\em $\sigma$-algebra of Borel sets}. If there is no confusion with the topology we will simply write $\Bi(X)$. All the measures considered in this section will be {\em non-negative finite Borel measures}, i.e. they could be the null measure, and since they will be defined on Polish spaces the finiteness condition will imply their {\em regularity} (see \cite[Proposition 8.1.12]{Cohn2013}). Given a (non-negative) finite Borel measure $\mu$ on a topological space $(X,\tau)$ we will denote its {\em support} by
\[
\supp(\mu) := X \setminus \bigcup_{ \underset{\mu(U)=0}{U \in \tau}} U.
\]
When $\mu$ is positive and regular it is easy to show that $\supp(\mu)$ is non-empty, and the smallest $\tau$-closed subset of $X$ with full measure, i.e. $\mu(\supp(\mu))=\mu(X)$, the later being true even if $\mu$ is not regular but $X$ is second-countable (see \cite[Proposition 2.3]{Kozarzewski2018}). Moreover, a point $x$ belongs to $\supp(\mu)$ if and only if $\mu(U)>0$ for every neighbourhood $U$ of $x$.\\[-5pt]

Before presenting the ``measures' constructing machine'' that will be used in the rest of this work, we give name to some properties that a Polish dynamical system $(X,T)$ may have. In particular, let $(X,\tau_X)$ be the underlying Polish space, $\tau$ a Hausdorff topology in $X$ and let $\Kc_{\tau}$ be the set of $\tau$-compact subsets of $X$. The properties that we are going to consider are the following:
\begin{enumerate}[(I)]
	\item $T$ is a continuous self-map of $(X,\tau)$ (i.e. $T:X\rightarrow X$ is $\tau$-continuous);
	
	\item $\tau \subset \tau_X$ (i.e. $\tau$ is coarser than $\tau_X$);
	
	\item $\Bi(X,\tau) = \Bi(X,\tau_X)$ (i.e. both topologies have the same Borel sets);
	
	\item every $\tau$-compact set is $\tau$-metrizable (i.e. every $K \in \Kc_{\tau}$ is $\tau$-metrizable);
	
	\item[(III*)] every point of $X$ has a neighbourhood basis for $\tau_X$ consisting of $\tau$-compact sets.
\end{enumerate}
In \cite[Fact 2.1]{GriMa2014} it is shown easily how (II) and (III*) imply conditions (III) and (IV). For the concrete recurrence results that we obtain, it is necessary to assume conditions (I), (II) and (III*) in order to use the reiteratively recurrent points in a successful way. However, without property (III*) and assuming just conditions (I), (II), (III) and (IV) we can carry out the ``construction'' on which everything is based:

\begin{lemma}[\textbf{Modification of \cite[Remarks 2.6 and 2.12]{GriMa2014}}]\label{Lem:NewRemark2.6and2.12}
	Let $(X,T)$ be a Polish dynamical system. Assume that $X$ is endowed with a Hausdorff topology $\tau$ which fulfills {\em(I)}, {\em(II)}, {\em(III)} and {\em(IV)}. Then for each $x_0 \in X$ and each Banach limit $\mf:\ell^{\infty}\rightarrow \RR$ one can find a (non-negative) $T$-invariant finite Borel regular measure $\mu$ on $X$ for which $\mu(X)\leq 1$ and such that
	\[
	\mu(K) \geq \mf(\1_{N(x_0,K)}) \quad \text{ for every } K \in \Kc_{\tau}.
	\]
	Moreover, we have the inclusion
	\[
	\supp(\mu) \subset \cl{\orb(x_0,T)}^{\tau}.
	\]
\end{lemma}

\begin{remark}
	Lemma \ref{Lem:NewRemark2.6and2.12} is a rather technical result which allows us to construct invariant measures. Note that:
	\begin{enumerate}[(a)]
		\item Assumptions (I), (II), (III) and (IV) are fulfilled by the initial topology $\tau_X$. However, if the $\tau_X$-compact sets are too small, given an arbitrary point $x_0 \in X$ (even with some kind of recurrence property) we could have $\mf(\1_{N(x_0,K)})=0$ for every $\tau_X$-compact set $K \subset X$ and hence the measure $\mu$ obtained could be the null measure on $X$. We will consider a strictly coarser topology $\tau \subsetneq \tau_X$ in order to obtain ``interesting measures'' from Lemma \ref{Lem:NewRemark2.6and2.12} (see Theorems \ref{The:Banach} and \ref{The:Hilbert}).
		
		\item Following the previous comment, even if the $\tau$-compact sets are big enough, the measure $\mu$ could be the null measure on $X$ if we choose a point $x_0 \in X$ for which the return sets $N(x_0,K)$ are too small and hence $\mf(\1_{N(x_0,K)})=0$ for every $K \in \Kc_{\tau}$. We will get ``interesting measures'' whenever we combine Lemma \ref{Lem:NewRemark2.6and2.12} together with the existence of a point $x_0 \in X$ and a Banach limit $\mf$ for which $\mf(\1_{N(x_0,K)})>0$ for some $\tau$-compact subsets $K$ of $X$. Those conditions will come from property (III*) together with the existence of a reiteratively recurrent point $x_0 \in \RRec(T)$, see Theorem~\ref{The:Main}.
		
		\item In the proof of \cite[Theorem 1.5]{GriMa2014} it is shown that under conditions (I), (II) and (III*), one can change the final statement of Lemma \ref{Lem:NewRemark2.6and2.12} into:\\[-5pt]
		
		{\em then for each $x_0 \in X$ one can find a $T$-invariant finite Borel measure $\mu$ on $X$ such that $\mu(K) \geq \dinf(N(x_0,K))$ for every $K \in \Kc_{\tau}$},\\[-5pt]
		
		simply by choosing a non-principal ultrafilter $\Uc$ on $\NN$ and considering the Banach limit
		\[
		\mf(\phi) := \lim_{\Uc} \frac{1}{n} \sum_{i=1}^n \phi(i) \quad \text{ for every } \phi \in \ell^{\infty}.
		\]
		Moreover, under the same assumptions it is also stated in \cite[Remark 2.12]{GriMa2014} that\\[-5pt]
		
		{\em for each $x_0 \in X$ and each $K \in \Kc_{\tau}$ one can find a $T$-invariant finite Borel measure $\mu$ on $X$ such that $\mu(K) \geq \dsup(N(x_0,K))$}.\\[-5pt]
		
		This just ensures that the above inequality holds true for only one fixed $\tau$-compact subset $K$ of $X$. We will encounter the same problem when working with the upper Banach density, and we will have to combine some more sophisticated Banach limits in order to cope with several $\tau$-compact sets at the same time, see Subsection~\ref{Subsec:2.3Main}.
	\end{enumerate}
\end{remark}

Here is the main result of this section:

\begin{theorem}\label{The:Main}
	Let $(X,T)$ be a Polish dynamical system. Assume that $X$ is endowed with a Hausdorff topology $\tau$ which fulfills {\em(I)}, {\em(II)}, and {\em(III*)}. If $x_0 \in X$ is a reiteratively recurrent point for $T$, then one can find a $T$-invariant probability measure $\mu_{x_0}$ on $X$ such that
	\[
	x_0 \in \supp(\mu_{x_0}) \subset \cl{\orb(x_0,T)}^{\tau}.
	\]
	Moreover, if $T$ is reiteratively recurrent then one can find a $T$-invariant probability measure $\mu$ on $X$ with full support.
\end{theorem}

\begin{remark}\label{Rem:compact}
	If the Polish dynamical system $T:(X,\tau_X)\rightarrow (X,\tau_X)$ is locally compact, its initial topology $\tau_X$ already fulfills properties (I), (II) and (III*), and hence (III) and (IV). In particular, the later is true whenever $(X,\tau_X)$ is a (metrizable) compact space.
\end{remark}

\subsection{Proof of Lemma \ref{Lem:NewRemark2.6and2.12}}

We modify the construction given in \cite[Section 2.2]{GriMa2014}. Let $(X,T)$ be a Polish dynamical system, denote by $\tau_X$ the initial topology of $X$ and assume that $X$ is endowed with a Hausdorff topology $\tau$ which fulfills (I), (II), (III) and (IV). Fix $x_0 \in X$ and let $\mf:\ell^{\infty}\rightarrow \RR$ be a Banach limit. For each $K \in \Kc_{\tau}$ denote by $\Ci(K,\tau)$ the space of all $\tau$-continuous real-valued functions on $K$.

\begin{fact}[\textbf{Modification of \cite[Fact 2.2]{GriMa2014}}]\label{Fact:2.2.1Riesz}
	For every $K \in \Kc_{\tau}$ there is a unique (non-negative) finite Borel regular measure $\mu_K$ on $K$ such that
	\[
	\int_K f d\mu_K = \int_{\NN} (\1_K f) (T^ix_0) d\mf(i) \quad \text{ for every } f \in \Ci(K,\tau).
	\]
	The measure $\mu_K$ satisfies $0\leq \mu_K(K) = \mf(\1_{N(x_0,K)}) \leq 1$.
\end{fact}
\begin{proof}
	The first part is obvious by the Riesz's Representation theorem since, as mentioned in \cite[Fact 2.2]{GriMa2014}, the formula
	\[
	L(f) := \int_{\NN} (\1_K f) (T^ix_0) d\mf(i) \quad \text{ for every } f \in \Ci(K,\tau),
	\]
	defines a (non-negative) linear functional on $\Ci(K,\tau)$. Moreover, the measure $\mu_K$ satisfies
	\[
	0 \leq \mu_K(K) = \int_{\NN} (\1_K) (T^ix_0) d\mf(i) = \mf(\1_{N(x_0,K)}) \leq \mf(\1) = 1.\qedhere
	\]
\end{proof}

By (III) we have the equality $\Bi(X,\tau) = \Bi(X,\tau_X)$ and hence for each $K \in \Kc_{\tau}$ we can extend the measure $\mu_K$ into a Borel measure on the whole space $X$ (still denoted by $\mu_K$) using the formula
\[
\mu_K(A) := \mu_K(K \cap A) \quad \text{ for every Borel set } A \in \Bi(X).
\]
Clearly $\mu_K(X) \leq 1$, which implies the regularity of these measures. However, since the compact sets $K \in \Kc_{\tau}$ are not necessarily $T$-invariant and we could have $T^{-1}(K) \cap K = \varnothing$, the measures $\mu_K$ are not necessarily $T$-invariant. As in \cite{GriMa2014} we will define the $T$-invariant measure we are looking for by taking the supremum of the measures $\mu_K$, and this is possible due to the following fact:

\begin{fact}[\textbf{\cite[Fact 2.3]{GriMa2014}}]\label{Fact:2.2.2Increasing}
	If $K,F \in \Kc_{\tau}$ and if $K \subset F$, then $\mu_K \leq \mu_F$.
\end{fact}
\begin{proof}
	The proof is exactly the same than that of \cite[Fact 2.3]{GriMa2014} and it uses essentially conditions (II), (IV) and the positivity of $\mf$.
\end{proof}

Since a finite union of $\tau$-compact subsets of $X$ is still an element of $\Kc_{\tau}$, from Fact~\ref{Fact:2.2.2Increasing} we deduce that the family $(\mu_K)_{K \in \Kc_{\tau}}$ has the following property: for any pair $K_1,K_2 \in \Kc_{\tau}$ one can find $F \in \Kc_{\tau}$ such that $\mu_F \geq \max\{\mu_{K_1},\mu_{K_2}\}$. Hence we can continue the construction given in \cite[Section~2.2]{GriMa2014}:

\begin{fact}[\textbf{Modification of \cite[Fact 2.4]{GriMa2014}}]\label{Fact:2.2.3muDef}
	If we set
	\[
	\mu(A) := \sup_{K\in\Kc_{\tau}} \mu_K(A) \quad \text{ for every Borel set } A \in \Bi(X),
	\]
	then $\mu$ is a (non-negative) Borel measure on $X$ such that $\mu(X) \leq 1$.
\end{fact}
\begin{proof}
	The proof is exactly the same than that of \cite[Fact 2.4]{GriMa2014} but with $\mu$ having the possibility of being the null measure.
\end{proof}

\begin{fact}[\textbf{Modification of \cite[Fact 2.5]{GriMa2014}}]\label{Fact:2.2.4muProp}
	The measure $\mu$ is $T$-invariant and we have the inequality
	\[
	\mu(K) \geq \mf(\1_{N(x_0,K)}) \quad \text{ for every } K \in \Kc_{\tau}.
	\]
\end{fact}
\begin{proof}
	The first part of the proof is exactly the same than that of \cite[Fact 2.5]{GriMa2014} and it uses essentially conditions (I), (IV), the positivity of $\mf$ and the fact that it is shift-invariant. By Fact \ref{Fact:2.2.1Riesz}, for each $K \in \Kc_{\tau}$ we have that
	\[
	\mu(K) \geq \mu_K(K) = \mf(\1_{N(x_0,K)}).\qedhere
	\]
\end{proof}

To finish the proof of Lemma \ref{Lem:NewRemark2.6and2.12} we include a property, not shown in \cite[Section 2.2]{GriMa2014}, about the support of the measure constructed:

\begin{fact}\label{Fact:2.2.5muSupp}
	We have the inclusion $\supp(\mu) \subset \cl{\orb(x_0,T)}^{\tau}$.
\end{fact}
\begin{proof}
	Write $O(x_0):=\cl{\orb(x_0,T)}^{\tau}$. First we show that for each $K \in \Kc_{\tau}$ we have the inclusion $\supp(\mu_K) \subset K \cap O(x_0)$: indeed, for any $K \in \Kc_{\tau}$ and any point $x \in K\setminus O(x_0)$, by compactness, there exists a positive function $f \in \Ci(K,\tau)$ such that
	\[
	f=0 \text{ on } K\cap O(x_0) \quad \text{ and } \quad f(x)=1. 
	\]
	If we suppose that $x \in \supp(\mu_K)$ and if we take a $\tau$-neighbourhood $U$ of $x$ in $(K,\tau)$ such that $f(U) \subset [1/2,\infty[$ then we have
	\[
	0 < \frac{1}{2}\mu_K(U) \leq \int_K f d\mu_K = \int_{\NN} (\1_K f) (T^ix_0) d\mf(i) = \mf(\symbf{0}) = 0,
	\]
	since $f(T^ix_0)=0$ for every $i\in\NN$, which is a contradiction. Hence $x \notin \supp(\mu_K)$.\\[-5pt]
	
	Finally, given $x \notin O(x_0)$ there is an $\tau$-open neighbourhood $U$ of $x$ in $(X,\tau)$, which by (II) is also a $\tau_X$-neighbourhood of $x$ in $(X,\tau_X)$, such that $U\cap O(x_0)=\varnothing$. Hence, since $\supp(\mu_K) \subset O(x_0)$ for every $K \in \Kc_{\tau}$, we deduce that $\mu_K(U)=\mu_K(K\cap U)=0$ for every $K \in \Kc_{\tau}$. By the definition of $\mu$ we get $\mu(U)=0$ and hence that $x \notin \supp(\mu)$.
\end{proof}

\subsection{Proof of Theorem \ref{The:Main}}\label{Subsec:2.3Main}

Let $(X,T)$ be a Polish dynamical system, denote by $\tau_X$ the initial topology of $X$ and assume that $X$ is endowed with a Hausdorff topology $\tau$ which fulfills (I), (II) and (III*).

\begin{fact}\label{Fact:2.3.1Bdsup}
	Given $x_0 \in X$ and $U \in \Kc_{\tau}$ with $\Bdsup(N(x_0,U))>0$, there exists a $T$-invariant probability measure $\mu$ on $X$ such that $\mu(U)>0$. Moreover, we have the inclusion
	\[
	\supp(\mu) \subset \cl{\orb(x_0,T)}^{\tau}.
	\]
\end{fact}

Fact~\ref{Fact:2.3.1Bdsup} allow us to (slightly) extend Theorem~\ref{The:Main} in terms of the {\em recurrence} notion introduced in \cite[Section~2.5]{GriMa2014} (see Remark~\ref{Rem:GriMa2014RRec} below for the explicit statement) which at the end turns out to be equivalent to the {\em recurrence} notion used here (see Theorem~\ref{The:Polish}).

\begin{proof}[Proof of Fact~\ref{Fact:2.3.1Bdsup}]
	Since
	\[
	\Bdsup(N(x_0,U)) := \limsup_{N \to \infty} \left( \max_{m \geq 0} \dfrac{\card(N(x_0,U)\cap [m,m+N])}{N+1} \right)>0,
	\]
	there exists an increasing sequence of natural numbers $(N_k)_{k \in \NN} \in \NN^{\NN}$ and a sequence of intervals $I_k = [i_k+1,i_k+N_k] \subset \NN$ such that
	\begin{equation}\label{eq:limU}
		\Bdsup(N(x_0,U)) = \lim_{k \to \infty} \dfrac{\card(N(x_0,U)\cap I_k)}{N_k}.
	\end{equation}
	Then we fix the Banach limit $\mf:\ell^{\infty}\rightarrow\RR$ defined as
	\[
	\mf(\phi) :=  \lim_{\Uc} \frac{1}{N_k} \sum_{n \in I_k} \phi(n) \quad \text{ for every } \phi \in \ell^{\infty},
	\]
	for some fixed non-principal ultrafilter $\Uc \subset \Part(\NN)$ on $\NN$. By \eqref{eq:limU} we have
	\[
	\mf(\1_{N(x_0,U)}) = \Bdsup(N(x_0,U))>0.
	\]
	Since $\tau$ fulfills (I), (II) and (III*), by \cite[Fact 2.1]{GriMa2014} it also has properties (III) and (IV) so we can apply Lemma \ref{Lem:NewRemark2.6and2.12} to $x_0$ and $\mf$ obtaining a (non-negative) $T$-invariant finite Borel measure $\mu$ on $X$ for which $\mu(K)\geq \mf(\1_{N(x_0,K)})$ for each $K \in \Kc_{\tau}$ and such that
	\[
	\supp(\mu) \subset \cl{\orb(x_0,T)}^{\tau}.
	\]
	In particular we get $\mu(U) \geq \mf(\1_{N(x_0,U)}) > 0$ so $\mu$ is a positive $T$-invariant finite Borel measure. Normalizing $\mu$ we get the desired measure.
\end{proof}

\begin{fact}\label{Fact:2.3.2mu_x0}
	Given $x_0 \in \RRec(T)$, there exists a $T$-invariant probability measure $\mu_{x_0}$ on $X$ such that
	\[
	x_0 \in \supp(\mu_{x_0}) \subset \cl{\orb(x_0,T)}^{\tau}.
	\]
\end{fact}
\begin{proof}
	Set $O(x_0):=\cl{\orb(x_0,T)}^{\tau}$. Using (III*), let $(U_n)_{n\in\NN}$ be a basis of $\tau_X$-neighbourhoods of $x_0$ consisting of $\tau$-compact sets. Applying Fact \ref{Fact:2.3.1Bdsup} to each set $U_n$ we obtain a sequence $(\mu_n)_{n\in\NN}$ of $T$-invariant probability measures on $X$ for which $\mu_n(U_n)>0$ and such that $\supp(\mu_n) \subset O(x_0)$ for each $n\in\NN$. Then the measure
	\[
	\mu_{x_0} := \sum_{n\in\NN} \frac{\mu_n}{2^n}
	\]
	is a $T$-invariant probability measure on $X$. Moreover, for any $\tau_X$-neighbourhood $U$ of $x_0$ there is an integer $n \in \NN$ with $U_n \subset U$ and hence
	\[
	\mu_{x_0}(U) \geq \mu_{x_0}(U_n) \geq \frac{\mu_n(U_n)}{2^n} > 0.
	\]
	This implies that $x_0 \in \supp(\mu_{x_0})$. Also, given $x \notin O(x_0)$ there is a $\tau$-neighbourhood $V$ of $x$, which by (II) is also a $\tau_X$-neighbourhood of $x$, such that $V\cap O(x_0)=\varnothing$. Since $\supp(\mu_n) \subset O(x_0)$ for every $n\in\NN$ we deduce that $\mu_n(V)=0$ for every $n \in \NN$ and by the definition of $\mu_{x_0}$ we get $\mu_{x_0}(V)=0$. This implies that $x \notin \supp(\mu_{x_0})$ and hence
	\[
	x_0 \in \supp(\mu_{x_0}) \subset O(x_0) = \cl{\orb(x_0,T)}^{\tau}.\qedhere
	\]
\end{proof}

To complete the proof of Theorem \ref{The:Main}, let $T$ be reiteratively recurrent. Since $X$ is separable there is a countable set $\{x_n:n\in\NN\} \subset \RRec(T)$ which is dense in $X$. Applying Fact \ref{Fact:2.3.2mu_x0} to each point $x_n$ we obtain a sequence $(\mu_{x_n})_{n\in\NN}$ of $T$-invariant probability measures on $X$ such that $x_n \in \supp(\mu_{x_n})$ for each $n \in \NN$. Finally, the measure
\[
\mu := \sum_{n\in\NN} \frac{\mu_{x_n}}{2^n}
\]
is a $T$-invariant probability measure on $X$ with full support.\QEDh

\begin{remark}\label{Rem:GriMa2014RRec}
	Under the initial topological assumptions of Theorem \ref{The:Main}, and in view of Fact \ref{Fact:2.3.1Bdsup}, a generalization in terms of the {\em recurrence} notion introduced in \cite[Section~2.5]{GriMa2014}, and following the spirit of \cite[Proposition 2.11]{GriMa2014}, can be shown:\\[-5pt]
	
	{\em If for each open subset $U$ of $X$ there is a point $x_U \in X$ such that $\Bdsup(N(x_U,U))>0$, then one can find a $T$-invariant probability measure $\mu$ on $X$ with full support.}\\[-5pt]
	
	Indeed, one just has to use (III*) to consider an appropriate countable family of $\tau$-compact sets whose $\tau_X$-interiors form a base of the initial topology $\tau_X$, apply Fact \ref{Fact:2.3.1Bdsup} to those $\tau$-compact sets and take an infinite convex combination of the obtained measures.
\end{remark}

\section{From reiterative to frequent recurrence}\label{Sec:3FRecfromRRec}

Theorem \ref{The:Main} allows us to construct invariant measures starting from reiteratively recurrent points. In this section, we exploit this result in order to show that reiterative recurrence for adjoint operators on separable dual Banach spaces actually implies the stronger notion of frequent recurrence (Theorem \ref{The:Banach}).

\subsection{A key lemma}

An important tool for the proof of Theorem \ref{The:Banach} is the following lemma:

\begin{lemma}[\textbf{Frequent Recurrence from Invariant Measures}]\label{Lem:Sophie}
	Let $T:X\rightarrow X$ be a continuous map on a second-countable space $X$ and let $\mu$ be a $T$-invariant probability measure on $X$. Then $\mu(\FRec(T))=1$ and in particular we have the inclusion
	\[
	\supp(\mu) \subset \cl{\FRec(T)}.
	\]
\end{lemma}

The above result is the recurrence version of \cite[Corollary 5.5]{BaMa2009}, and since recurrence is a local property the measure is not required to be with full support, condition under which the map $T$ would clearly be frequently recurrent.

\begin{proof}[Proof of Lemma \ref{Lem:Sophie}]
	Let $B \in \Bi(X)$ be an arbitrary but fixed Borel set with $\mu(B)>0$. By the Ergodic Decomposition theorem (see \cite[Theorem 3.42]{Glasner2003}) there is a $T$-invariant probability measure $m$ on $X$ for which $T$ is an ergodic map and such that $m(B)>0$. Let $(U_n)_{n\in\NN}$ be a countable basis of the topology and apply the Birkhoff's Pointwise Ergodic theorem (see \cite[Theorem~3.41]{Glasner2003}) to each of the indicator functions $\1_{U_n}$. This yields
	\begin{eqnarray}
		\dens(N(x,U_n)) &=& \lim_{N \to \infty} \frac{\card(N(x,U_n) \cap [0,N])}{N+1} = \lim_{N \to \infty} \frac{1}{N+1} \sum_{k=0}^{N} \1_{U_n}(T^kx) \nonumber\\[10pt]
		&=& \int_X \ \1_{U_n} \ dm = m(U_n), \nonumber
	\end{eqnarray}
	for $m$-a.e. point $x \in X$, that is, for each $n \in \NN$ there is a set $A_n \subset X$ with $m(A_n)=1$ such that $\dens(N(x,U_n))=m(U_n)$ for every $x \in A_n$. Since a countable union of null sets is again null, the set
	\[
	A := \supp(m) \cap \left(\bigcap_{n\in\NN} A_n \right)
	\]
	satisfies $m(A)=1$. We claim that $A \subset \FRec(T)$. Indeed, for every $x \in A$ and every neighbourhood $U$ of $x$ there is an integer $n\in\NN$ such that $x \in U_n \subset U$. Since $A \subset \supp(m)$ we have that $U_n \cap \supp(m)\neq\varnothing$ and hence
	\[
	\dinf(N(x,U)) \geq \dinf(N(x,U_n)) = m(U_n) > 0.
	\]
	The arbitrariness of the neighbourhood $U$ of $x$ implies that $x \in \FRec(T)$. Now, since $m(A)=1$ and $m(B)>0$ we obtain $A\cap B\neq\varnothing$ and hence
	\[
	\FRec(T)\cap B \neq \varnothing.
	\]
	Since this is true for every set $B \in \Bi(X)$ with $\mu(B)>0$ we deduce that $\mu(\FRec(T))=1$. Then $\mu(\cl{\FRec(T)})=1$ and in particular, since $\supp(\mu)$ is the smallest closed subset of $X$ with full $\mu$-measure, we get that
	\[
	\supp(\mu) \subset \cl{\FRec(T)}.\qedhere
	\]
\end{proof}

\begin{remark}
	Lemma~\ref{Lem:Sophie} improves \cite[Theorem~3.3]{Furstenberg1981} in terms of frequent recurrence by using the Birkhoff's Pointwise Ergodic theorem. Indeed, under the assumptions of Lemma~\ref{Lem:Sophie}, \cite[Theorem~3.3]{Furstenberg1981} shows that $\mu$-a.e. point is recurrent, i.e. $\mu(\Rec(T))=1$.
\end{remark}

Combining Theorem \ref{The:Main} and Lemma \ref{Lem:Sophie} we deduce the following result:

\begin{theorem}[\textbf{From Reiterative to Frequent Recurrence}]\label{The:Polish}
	Let $(X,T)$ be a Polish dynamical system, denote by $\tau_X$ the initial topology of $X$ and assume that $X$ is endowed with a Hausdorff topology $\tau$ which fulfills {\em(I)}, {\em(II)}, and {\em(III*)}. Then we have the equality
	\[
	\cl{\FRec(T)}^{\tau_X} = \cl{\RRec(T)}^{\tau_X}.
	\]
	Moreover:
	\begin{enumerate}[{\em(a)}]
		\item The following statements are equivalent:
		\begin{enumerate}[{\em(i)}]
			\item $\FRec(T) \neq \varnothing$;
			
			\item $\UFRec(T) \neq \varnothing$;
			
			\item $\RRec(T) \neq \varnothing$;
			
			\item $T$ admits an invariant probability measure.
		\end{enumerate}
		
		\item The following statements are equivalent:
		\begin{enumerate}[{\em(i)}]
			\item $T$ is frequently recurrent;
			
			\item $T$ is $\Uc$-frequently recurrent;
			
			\item $T$ is reiteratively recurrent;
			
			\item $T$ admits an invariant probability measure with full support.
		\end{enumerate}
	\end{enumerate}
\end{theorem}
\begin{proof}
	By definition we always have $\FRec(T) \subset \UFRec(T) \subset \RRec(T)$, so we just have to show that
	\[
	\RRec(T) \subset \cl{\FRec(T)}^{\tau_X}.
	\]
	Suppose that $\RRec(T)\neq\varnothing$. Given any $x_0 \in \RRec(T)$, by Theorem \ref{The:Main} one can find a $T$-invariant probability measure $\mu_{x_0}$ on $X$ for which $x_0 \in \supp(\mu_{x_0})$. Since separable and metrizable spaces are second-countable, Lemma \ref{Lem:Sophie} implies that $x_0 \in \cl{\FRec(T)}^{\tau_X}$.\\[-5pt]
	
	Moreover, in both cases (a) and (b) we have: (i) implies (ii) which implies (iii) by definition; (iii) implies (iv) by Theorem \ref{The:Main}; and (iv) implies (i) by Lemma \ref{Lem:Sophie}.
\end{proof}

As we already mentioned in the Introduction, this result is false for general Polish dynamical systems: there exist even reiteratively hypercyclic operators on $c_0(\NN)$ without any non-zero $\Uc$-frequently recurrent vector (see \cite[Theorem 5.7 and Corollary 5.8]{BoGrLoPe2020}). Working with linear dynamical systems implies a reformulation of the above result, which is Theorem~\ref{The:Banach}.

\subsection{Proof of Theorem \ref{The:Banach}}

Let $T:X\rightarrow X$ be an adjoint operator on a separable dual Banach space $X$. Denote by $\tau_{\|\cdot\|}$ the norm topology, consider the weak-star topology $w^*$ and note that:
\begin{enumerate}[(I)]
	\item since $T$ is an adjoint operator, it is a continuous self-map of $(X,w^*)$;
	
	\item by the definition of the topologies, we have $w^* \subset \tau_{\|\cdot\|}$;
	
	\item[(III*)] by the Alaoglu-Bourbaki's theorem, the translation of the family of closed balls centred at $0$ is a $\tau_{\|\cdot\|}$-neighbourhood basis consisting of $w^*$-compact sets. 
\end{enumerate}
If $T:X\rightarrow X$ is an operator on a separable reflexive Banach space $X$ the same conditions hold for the weak topology. From here one can apply the same arguments as those used in the proof of Theorem~\ref{The:Polish}. In particular, if we consider a point $x_0 \in \RRec(T)\setminus\{0\}$ then the measure $\mu_{x_0}$ obtained by Theorem \ref{The:Main} is a non-trivial invariant probability measure.\QEDh

\begin{remark}\label{Rem:Non-SeparableI}
	The equality $\cl{\FRec(T)} = \cl{\RRec(T)}$ and hence the equivalences (i) $\Leftrightarrow$ (ii) $\Leftrightarrow$ (iii) established in Theorem \ref{The:Banach} are still true when the underlying space $X$ is a non-separable reflexive Banach space. Indeed, given an operator $T:X\rightarrow X$ on a non-separable reflexive Banach space $X$, and given a point $x_0 \in \RRec(T)$ we can consider the separable closed $T$-invariant subspace
	\[
	Z := \cl{\lspan(\orb(x_0,T))},
	\]
	which is again reflexive. Then $T\res_Z : Z \rightarrow Z$ is an operator on a separable reflexive Banach space. Moreover, recurrence is a local property, i.e. for each Furstenberg family $\Fc$ we have the equality:
	\[
	\Fc\Rec(T\res_Z) = \Fc\Rec(T) \cap Z.
	\] 
	Applying Theorem \ref{The:Banach} to $T\res_Z$ we have
	\[
	x_0 \in \RRec(T\res_Z) \text{ and hence } x_0 \in \cl{\FRec(T\res_Z)} \subset \cl{\FRec(T)}.
	\]
	However, we cannot say the same about statement (iv) of Theorem \ref{The:Banach} since separability is essential to construct and extend the invariant measures onto the whole space. The above arguments are also restricted to the reflexive case because closed subspaces of a dual Banach space are not necessarily dual Banach spaces (consider $c_0(\NN) \subset \ell^{\infty}(\NN)$).
\end{remark}

\section{From uniform recurrence to unimodular eigenvectors}\label{Sec:4EigfromURec}

Our aim in this section is to connect some recurrence properties (stronger than those considered in Sections \ref{Sec:2IMfromRRec} and \ref{Sec:3FRecfromRRec}), for linear dynamical systems on {\em complex} Banach spaces, to the existence of unimodular eigenvectors. This investigation is motivated by the fact that, given a {\em complex linear map} $T:X\rightarrow X$ on a {\em complex topological vector space} $X$, the linear span of its unimodular eigenvectors $\Ec(T)$ consists of $\Del^*$-recurrent vectors. It is shown in \cite[Lemma~7.1 and Corollary~7.2]{BoGrLoPe2020} that they are $\IP^*$-recurrent, and in fact, the same arguments hold by using \cite[Proposition 9.8]{Furstenberg1981} applied to the Kronecker system consisting of the compact group $\TT^k$ and the (left) multiplication $(z_1,...,z_k) \mapsto (\lambda_1z_1,...,\lambda_kz_k)$ for a fixed $k$-tuple $(\lambda_1,...,\lambda_k) \in \TT^k$. We give an alternative proof via invariant measures:

\begin{proposition}\label{Pro:unimodular}
	Let $T:X\rightarrow X$ be a complex linear map on a complex topological vector space $X$. A linear combination
	of unimodular eigenvectors $\Ec(T)$ is a $\Del^*$-recurrent vector, i.e. $\lspan(\Ec(T)) \subset \Del^*\Rec(T)$.
\end{proposition}
\begin{proof}
	Given $\lambda \in \TT$ let $R_{\lambda}:\TT \rightarrow \TT$ be the $\lambda$-rotation map where $z \mapsto \lambda z$. Then given $\eps>0$, since the Haar measure on $\TT$ is a $R_{\lambda}$-invariant measure with full support, by the Poincar\'e's Recurrence theorem (see \cite[Theorem 3.2 and Page 177]{Furstenberg1981}) there is $A \in \Del^*$ such that for the set $B(1,\eps/2) := \{ z \in \TT : |1-z|<\eps/2 \}$ we have
	\[
	R_{\lambda}^n(B(1,\eps/2)) \cap B(1,\eps/2) \neq \varnothing \text{ for every } n \in A.
	\]
	By the triangular inequality we get $|\lambda^n-1| < \eps$ for each $n \in A$ and hence
	\[
	\Del^* \ni A \subset \{ n \in \NN : |\lambda^n-1|<\eps \} \text{ so } \{ n \in \NN : |\lambda^n-1|<\eps \} \in \Del^*.
	\]
	Since the Furstenberg family $\Del^*$ is a filter (see \cite{BerDown2008}) and $\lambda \in \TT$ and $\eps>0$ were chosen arbitrarily the proof is finished.
\end{proof}

Hence, given a complex linear dynamical system $T:X\rightarrow X$ we will always have:
\[
\lspan(\Ec(T)) \subset \Del^*\Rec(T) \subset \IP^*\Rec(T) \subset \URec(T) \subset \RRec^{bo}(T).
\]
Our goal is now to prove Theorem \ref{The:Hilbert}, which states that for any operator acting on a complex Hilbert space, the existence of a non-zero reiteratively recurrent vector with bounded orbit, and in particular the existence of a uniformly recurrent vector, implies the existence of a unimodular eigenvector. The proof of Theorem~\ref{The:Hilbert} relies heavily on the machinery of {\em Gaussian measures} on (complex separable) Hilbert spaces. We begin by recalling some basic facts concerning these Gaussian measures, as well as some deeper results pertaining to the Ergodic Theory of Gaussian linear dynamical systems. We refer the reader to one of the references \cite{ChoTarVak1987} or \cite{DiJaTo1995} for more about Gaussian measures on Banach spaces, and to \cite{BaMa2009} and \cite{BaMa2016} for more on their role in linear dynamics.

\subsection{Ergodic Theory for linear dynamical systems and Gaussian measures}

The study of Ergodic Theory in the framework of linear dynamics started with the pioneering work of Flytzanis (see \cite{Flyztanis1994,Flyztanis1995}), and was then further developed in the papers \cite{BaGri2006}, \cite{BaGri2007} and \cite{BaMa2016}, among others, focusing on the existence of invariant {\em Gaussian} measures satisfying some further dynamical properties such as weak/strong mixing.

\begin{definition}
	A Borel probability measure $m$ on a complex Banach space $X$ is said to be a {\em Gaussian measure} if every continuous linear functional $x^* \in X^*$ has a complex Gaussian distribution when considered as a random variable on $(X,\Bi(X),m)$.
\end{definition}

It is now well understood that the dynamics of a linear dynamical system $(X,T)$ are closely related to the properties of the unimodular eigenvectors of $T$. The situation is especially well understood in the Hilbertian setting, since the existence of an invariant Gaussian measure (with full support, or with respect to which $T$ is ergodic or weakly/strongly mixing) can be fully characterized in terms of the properties of the set $\Ec(T)$. See \cite{BaGri2006} and~\cite{BaMa2009} for details. These characterizations do not hold true, in general, in the Banach space setting, but still many results are preserved allowing for a rather through understanding of Ergodic Theory of linear dynamical systems in this Gaussian framework. See \cite{BaGri2007}, \cite{BaMa2009} and \cite{BaMa2016} for details. Even though Gaussian measures are an essential tool for our proof of Theorem~\ref{The:Hilbert} (see Lemma~\ref{Lem:Eigenvectors} below), the properties that such measures (may) have are properties that arbitrary probability measures can have too. We introduce these properties following \cite{ChoTarVak1987}:

\begin{definition}
	Let $\mu$ be a probability measure on a Banach space $X$:
	\begin{enumerate}[(a)]
		\item suppose that there exists an element $x \in X$ such that
		\[
		\int_X \ep{x^*}{z} d\mu(z) = \ep{x^*}{x} \quad \text{ for every } x^* \in X^*,
		\]
		then $x$ is called the {\em expectation}
		of the measure $\mu$, and in this case we will write
		\[
		\int_X z d\mu(z) := x;
		\]
		
		\item we say that $\mu$ is {\em centered} if its expectation exists and it is equal to $0 \in X$;
		
		\item we say that $\mu$ has a {\em finite second-order moment}, if
		\[
		\int_X \|z\|^2 d\mu(z) < \infty.
		\]
	\end{enumerate}	
\end{definition} 

If $\mu$ has a finite second-order moment then its expectation (called the {\em Pettis integral} of~$\mu$) exists (see \cite[Page 55]{DiUh1977}). Given a centered probability measure $\mu$ on $X$ with a finite second-order moment, following \cite[Page~169]{ChoTarVak1987} and \cite[Theorem~5.9]{BaMa2009}, we can define the {\em covariance operator} of such a measure $\mu$ as the bounded linear operator $R:X^*\rightarrow X$ satisfying
\[
\ep{y^*}{Rx^*} = \int_X \ep{y^*}{z} \ep{x^*}{z} d\mu(z)
\]
for every pair of elements $x^*$ and $y^*$ of $X^*$. In other words,
\begin{equation}\label{eq:covarianceR}
	Rx^* := \int_X \ep{x^*}{z}z d\mu(z) \quad \text{ for every } x^* \in X^*.
\end{equation}
Any Gaussian measure $m$ on $X$ has a finite second-order moment (see \cite[Exercise~5.5]{BaMa2009}), and since we will consider in this work only centered Gaussian measures, we will always have an associated covariance operator for such a measure $m$.\\[-5pt]

When $H$ is a complex separable Hilbert space, the {\em covariance operator} of a centered probability measure $\mu$ on $H$ with a finite second-order moment is usually defined, in a slightly different way, as the bounded linear operator $S:H\rightarrow H$ for which
\[
\ep{Sx}{y} = \int_H \ep{x}{z} \cl{\ep{y}{z}} d\mu(z) \quad \text{ for every } x,y \in H,
\]
i.e.
\begin{equation}\label{eq:covarianceS}
	Sx := \int_H \ep{x}{z}z d\mu(z) \quad \text{ for every } x \in H.
\end{equation}
Observe that, contrary to \eqref{eq:covarianceR}, in this case $\ep{Sx}{\cdot}:H\rightarrow \CC$ is an anti-linear functional acting on $H$. Also, $S$ is a self-adjoint positive trace-class operator on $H$. It is a standard result (see for instance \cite[Corollary 5.15]{BaMa2009}) that the Gaussian covariance operators on $H$ are exactly the positive trace-class operators on $H$, i.e. for such an operator $S$ there exists a Gaussian measure $m$ on $H$ for which we also have that
\[
\ep{Sx}{y} = \int_H \ep{x}{z} \cl{\ep{y}{z}} dm(z) \quad \text{ for every } x,y \in H.
\]
The possibility of constructing a Gaussian measure $m$ with the same covariance operator as $\mu$, together with the fact that the support of the Gaussian measure $m$ is exactly $\cl{S(H)}$ (i.e. the closed linear span of the range of its covariance operator, see \cite[Proposition 5.18]{BaMa2009}) is the key to prove the following lemma, inspired from the pioneering work \cite{Flyztanis1995} of Flytzanis. This lemma is crucial for the proof of Theorem~\ref{The:Hilbert}.

\begin{lemma}[\textbf{Unimodular Eigenvectors from Invariant Measures}]\label{Lem:Eigenvectors}
	Let $T \in \Lc(H)$, where $H$ is a complex separable Hilbert space, and let $\mu$ be a (non-trivial) $T$-invariant probability measure on $H$ such that $\int_H \|z\|^2 d\mu(z) < \infty$. Then we have the inclusions
	\[
	\supp(\mu) \subset \cl{\lspan(\supp(\mu))} \subset \cl{\lspan(\Ec(T))}.
	\]
\end{lemma}
\begin{proof}
	Suppose first that $\mu$ is a centered measure on $H$. Then, since $H$ is a Hilbert space, the covariance operator $S$ of $\mu$ defined as in \eqref{eq:covarianceS} satisfies
	\[
	\ep{Sx}{y} = \int_H \ep{x}{z} \cl{\ep{y}{z}} d\mu(z) = \int_{\supp(\mu)} \ep{x}{z} \cl{\ep{y}{z}} d\mu(z) \quad \text{ for every } x,y \in H,
	\]
	and by \cite[Corollary 5.15]{BaMa2009}, it is also the covariance operator of a certain Gaussian measure $m$ on $H$. From now on we split the proof in three steps:	
	\begin{enumerate}[{{\em Step}} 1.]
		\item \textit{The Gaussian measure $m$ is $T$-invariant}:
		
		Given $x,y \in H$ we have
		\begin{eqnarray}
			\ep{TST^*x}{y} &=& \ep{ST^*x}{T^*y} = \int_H \ep{T^*x}{z} \cl{\ep{T^*y}{z}} d\mu(z) = \int_H \ep{x}{Tz} \cl{\ep{y}{Tz}} d\mu(z)  \nonumber\\[10pt]
			&=& \int_H \ep{x}{z} \cl{\ep{y}{z}} d(\mu\circ T^{-1})(z) = \int_H \ep{x}{z} \cl{\ep{y}{z}} d\mu(z) = \ep{Sx}{y}, \nonumber
		\end{eqnarray}
		since $\mu$ is $T$-invariant. By \cite[Proposition 5.22]{BaMa2009} we deduce that $m$ is $T$-invariant.
		
		\item \textit{We have the equality $\cl{\lspan(\supp(\mu))} = \supp(m)$}:
		
		By \cite[Proposition 5.18]{BaMa2009} we know that $\supp(m) = \ker(S)^{\perp} = \cl{S(H)}$. Moreover, the subspace $\cl{\lspan(\supp(\mu))}^{\perp}$ is included in the set
		\[
		\hspace{-2.1cm}\left\{ y \in H : \ep{Sx}{y} = \int_{\supp(\mu)} \ep{x}{z} \cl{\ep{y}{z}} d\mu(z) = 0 \text{ for every } x \in H \right\} = \\[-12.5pt]
		\]
		\begin{eqnarray}
			&=& \cl{S(H)}^{\perp} \subset \left\{ y \in H : \int_{\supp(\mu)} |\ep{y}{z}|^2 d\mu(z) = 0 \right\} \nonumber\\[7.5pt]
			&=& \left\{ y \in H : \ep{y}{z} = 0 \text{ for } \mu\text{-a.e. } z \in H \right\} \nonumber\\[7.5pt]
			&\overset{(*)}{=}& \left\{ y \in H : \ep{y}{z} = 0 \text{ for every } z \in \supp(\mu) \right\} = \cl{\lspan(\supp(\mu))}^{\perp}, \nonumber
		\end{eqnarray}
		where the equality $(*)$ follows from the continuity of the maps $\ep{y}{\cdot}:H\rightarrow\CC$.
		
		\item \textit{We have the inclusion $\supp(\mu) \subset \cl{\lspan(\Ec(T))}$}:
		
		In \cite[Theorem 5.46]{BaMa2009} it is stated that {\em if a Banach space $X$ has cotype 2, then every operator in $\Lc(X)$ admitting a Gaussian invariant measure with full support has a spanning set of unimodular eigenvectors}. Since the support of $m$ is a closed linear subspace of $H$ ({\em Step} 2) and every Hilbert space has cotype 2, \cite[Theorem 5.46]{BaMa2009} applied to the $T$-invariant measure $m$ ({\em Step} 1) implies that $\supp(m) \subset \cl{\lspan(\Ec(T))}$, and hence using again {\em Step} 2 we get that
		\[
		\supp(\mu) \subset \cl{\lspan(\supp(\mu))} = \supp(m) \subset \cl{\lspan(\Ec(T))}.
		\]
	\end{enumerate}
	
	Suppose now that $\mu$ is not centered and define the measure
	\[
	\nu(A) := \int_{\TT} \mu(\lambda A) d\lambda \quad \text{ for every Borel set } A \in \Bi(H). 
	\]
	Then $\nu$ is a (non-trivial) probability measure on $H$ and it is $T$-invariant since
	\[
	\nu(T^{-1}(A)) = \int_{\TT} \mu(\lambda T^{-1}(A)) d\lambda = \int_{\TT} \mu(T^{-1}(\lambda A)) d\lambda = \nu(A).
	\]
	Using the density of the simple functions in $L^1(H,\Bi(H),\nu)$ one can show that $\nu$ is centered since
	\[
	\int_H zd\nu(z) = \int_{\TT} \left(\int_H \cl{\lambda} z d\mu(z) \right) d\lambda = \int_{\TT} \cl{\lambda} \left(\int_H z d\mu(z) \right) d\lambda = 0,
	\]
	and also that $\nu$ has a finite second-order moment since
	\[
	\int_H \|z\|^2 d\nu(z) = \int_{\TT} \left(\int_H \|\cl{\lambda}z\|^2 d\mu(z) \right) d\lambda = \int_H \|z\|^2 d\mu(z) < \infty.
	\]
	The first part of the proof implies that $\supp(\nu) \subset \cl{\lspan(\supp(\nu))} \subset \cl{\lspan(\Ec(T))}$ and hence we only have to show that $\supp(\mu) \subset \supp(\nu)$. In order to see this, pick $x_0 \in \supp(\mu)$ and $\eps>0$. Then let $\delta:=\mu(B(x_0,\eps/2))>0$, where $B(x_0,\eps/2)$ denotes the open ball of $X$ centred at $x_0$ and of radius $\eps/2$, and note that $B(x_0,\eps/2) \subset \lambda B(x_0,\eps)$ for any $\lambda \in \TT$ with
	\[
	|\lambda-1| < \frac{\eps}{2(\|x_0\|+1)}.
	\]
	Indeed, given $x \in B(x_0,\eps/2)$ we have that $\|\lambda x_0-x\| \leq \|(\lambda-1)x_0\| + \|x_0-x\| < \eps$. Then for such a $\lambda \in \TT$ we have that $\mu(\lambda B(x_0,\eps)) \geq \delta$ and hence $\nu(B(x_0,\eps))>0$. The arbitrariness of $\eps>0$ implies that $x_0 \in \supp(\nu)$.
\end{proof}

\begin{remark}\label{Rem:co-type2BlackBox}
	If we start the proof of Lemma \ref{Lem:Eigenvectors} with the underlying space being a Banach space $X$ which has type 2, then there exists a Gaussian measure $m$ on $X$ whose covariance operator is $R$, as defined in \eqref{eq:covarianceR}.	Indeed, since $R$ is a symmetric and positive operator it admits a square root: there exist some separable Hilbert space $H$ and an operator $K:H\rightarrow X$ such that $R=KK^*$ (see \cite[Page 101]{BaMa2009}). Moreover, by the finite second-order moment condition of $\mu$, the operator $K^*$ is an absolutely 2-summing operator and hence such a Gaussian measure $m$ on $X$ exists by \cite[Corollary 5.20]{BaMa2009}. However, in the {\em Step} 3 of the proof above the underlying space needs to have cotype 2. Since the only spaces which are both of type 2 and of cotype 2 are those which are isomorphic to a Hilbert space, the proof of Lemma~\ref{Lem:Eigenvectors} does not extend outside of the Hilbertian setting.
\end{remark}

We are now ready to prove Theorem \ref{The:Hilbert}.

\subsection{Proof of Theorem \ref{The:Hilbert}}

Let $T:H\rightarrow H$ be an operator on a complex separable Hilbert space $H$. We already know that $\lspan(\Ec(T)) \subset \URec(T) \subset \RRec^{bo}(T)$, so we just have to prove that
\[
\RRec^{bo}(T) \subset \cl{\lspan(\Ec(T))}.
\]
To see this, let $x_0 \in \RRec^{bo}(T)\setminus\{0\}$ and let $M>0$ be such that $\orb(x_0,T)$ is contained in $MB_H$, the $\|\cdot\|$-closed ball of radius $M$ centred at $0$. If we denote by $w$ the weak topology of $H$, we have the inclusion
\[
\cl{\orb(x_0,T)}^w \subset MB_H.
\]
By Theorem \ref{The:Main} there is a (non-trivial, because $x_0\neq 0$) $T$-invariant probability measure $\mu_{x_0}$ on $X$ such that
\[
x_0 \in \supp(\mu_{x_0}) \subset \cl{\orb(x_0,T)}^w,
\]
and hence
\[
\int_H \|z\|^2 d\mu_{x_0}(z) = \int_{\supp(\mu_{x_0})} \|z\|^2 d\mu_{x_0}(z) \leq M^2 < \infty.
\]
By Lemma~\ref{Lem:Eigenvectors} we get that $x_0 \in \supp(\mu_{x_0}) \subset \cl{\lspan(\Ec(T))}$ as we wanted to show.\\[-5pt]

Suppose now that there is a countable set $\{x_n:n\in\NN\} \subset \RRec^{bo}(T)$ which is dense in $H$. For each $n \in \NN$ pick $M_n>0$ and $k_n \in \NN$ such that
\[
\cl{\orb(x_n,T)}^w \subset M_nB_H \quad \text{ and } \quad 2^nM_n^2 \leq 2^{k_n}.
\]
Applying Theorem \ref{The:Main} to each vector $x_n$ we obtain a sequence $(\mu_{x_n})_{n\in\NN}$ of $T$-invariant probability measures on $H$ such that $x_n \in \supp(\mu_{x_n}) \subset \cl{\orb(x_n,T)}^w$ for each $n \in \NN$. Consider the measure
\[
\mu := \sum_{n\in\NN} \frac{\mu_{x_n}}{2^{k_n}},
\]
which is a (positive) $T$-invariant finite Borel measure on $H$ with full support such that
\[
\int_H \|z\|^2 d\mu(z) = \sum_{n\in\NN} \frac{1}{2^{k_n}} \int_H \|z\|^2 d\mu_{x_n}(z) \leq \sum_{n\in\NN} \frac{M_n^2}{2^{k_n}} \leq 1.
\]
Normalizing $\mu$ we get a $T$-invariant probability measure with full support and finite second-order moment.\\[-5pt]

Moreover, in both cases (a) and (b) we have: (i) $\Rightarrow$ (ii) $\Rightarrow$ (iii) $\Rightarrow$ (iv) $\Rightarrow$ (v) $\Rightarrow$ (vi) $\Rightarrow$ (vii) by definition; (vii) implies (viii) using Theorem \ref{The:Main} as in the above arguments; and (viii) implies (i) by Lemma \ref{Lem:Eigenvectors}.\QEDh

\begin{remark}\label{Rem:Non-SeparableII}
	The equalities $\cl{\lspan(\Ec(T))} = \cl{\URec(T)} = \cl{\RRec^{bo}(T)}$ and hence the equivalences (i) $\Leftrightarrow$ (ii) $\Leftrightarrow$ (iii) $\Leftrightarrow$ (iv) $\Leftrightarrow$ (v) $\Leftrightarrow$ (vi) $\Leftrightarrow$ (vii) established in Theorem \ref{The:Hilbert} are still true when the underlying space $H$ is a complex non-separable Hilbert space. Since the closed subspaces of a Hilbert space are again Hilbert spaces, the same arguments as those used in Remark \ref{Rem:Non-SeparableI} apply. We loose again the measures equivalences, i.e. statement (viii).
\end{remark}

As mentioned in the Introduction and in Remark \ref{Rem:co-type2BlackBox}, the proof of Lemma \ref{Lem:Eigenvectors} and hence that of Theorem \ref{The:Hilbert} do not extend outside of the Hilbertian setting.\\[-5pt]

We finish this section with the proof of Theorem \ref{The:JaDeGli+URec}, which concerns the power-bounded operators on complex reflexive Banach spaces $X$. The proof relies on the splitting theorem of Jacobs-Deleeuw-Glicksberg, and is really specific to the setting of power-bounded operators. We follow the presentation and notation of \cite[Section 2.4]{Krengel1985}: if $\mathscr{S}$ is a semigroup of $\Lc(X)$, we say that $\mathscr{S}$ is {\em weakly almost periodic} if for any $x \in X$ the set $\mathscr{S}x = \{ Sx : S \in \mathscr{S} \}$ has a $w$-compact closure.

\subsection{Proof of Theorem \ref{The:JaDeGli+URec}}

Given a power-bounded operator $T:X\rightarrow X$ on a complex reflexive Banach space $X$, we already know that $\lspan(\Ec(T)) \subset \URec(T)$ so we just have to show the inclusion $\URec(T) \subset \cl{\lspan(\Ec(T))}$. We set $O(x) := \cl{\orb(x,T)}^{w}$ for each $x \in X$.\\[-5pt]

Since $T$ is power-bounded, every $T$-orbit is bounded and has a $w$-compact closure. Hence, by the Jacobs-Deleeuw-Glicksberg theorem \cite[Section 2.4, Theorem 4.4]{Krengel1985} applied to the (weakly almost periodic) abelian semigroup of operators $\{ T^n : n\in\NN_0 \} \subset \Lc(X)$, we obtain that $X = X_{rev} \oplus X_{fl}$ where
\[
X_{rev} := \{ x \in X : y \in O(x) \Rightarrow x \in O(y) \} \quad \text{ and } \quad X_{fl} := \{ x \in X : 0 \in O(x) \}.
\]
Moreover, by the second part of this same theorem \cite[Section 2.4, Theorem~4.5]{Krengel1985} we also get that
\[
X_{rev} = \cl{\lspan(\Ec(T))}.
\]
Let us now show that $\URec(T) \subset X_{rev}$. Indeed, given $x \in \URec(T) \setminus\{0\}$ we can consider the map $T\res_{O(x)}:(O(x),w)\rightarrow (O(x),w)$ which is a $w$-compact dynamical system. Since the weak topology is coarser than the norm topology we have $x \in \URec(T\res_{O(x)})$ and hence by \cite[Theorem 1.17]{Furstenberg1981} the system $T\res_{O(x)}$ is minimal so every $T\res_{O(x)}$-orbit is dense in $O(x)$. Finally, given $y \in O(x)$ we have
\[
O(y) = \cl{\orb(y,T)}^{w} = \cl{\orb(y,T\res_{O(x)})}^{w} = O(x)
\]
which implies that $x \in O(y)$. The arbitrariness of $y \in O(x)$ shows that $x \in X_{rev}$.\QEDh

\section{Product dynamical systems}\label{Sec:5Product}

Given a property of a dynamical system $T:X\rightarrow X$, it is usual to ask whether the product dynamical system $T\times T:X\times X \rightarrow X\times X$ has the same property. Studied cases in linear dynamics are transitivity or hypercyclicity (which gives us the concept of topological weak mixing), and in general $\Fc$-transitivity or $\Fc$-hypercyclicity (see \cite{BMPP2019} and \cite{ErEsMe2021}). Here we show that the above theorems still work for the product systems.

\begin{theorem}[\textbf{From Reiterative to $\symbf{N}$-Dimensional Frequent Recurrence}]
	Let $N\in\NN$ and suppose that for each $1\leq i\leq N$ there is a Polish dynamical system $(X_i,T_i)$ such that $(X_i,\tau_{X_i})$ can be endowed with a Hausdorff topology $\tau_i$ which fulfills {\em(I)}, {\em(II)}, and {\em(III*)} with respect to the map $T_i$ and the topology $\tau_{X_i}$. Then for the product dynamical system $T:(X,\tau_X)\rightarrow (X,\tau_X)$, where $\tau_X$ is the product topology of the $N$-th $\tau_{X_i}$ topologies, we have the equality
	\[
	\cl{\FRec(T)}^{\tau_X} = \prod_{i=1}^N \cl{\RRec(T_i)}^{\tau_{X_i}}.
	\]
	In particular:
	\begin{enumerate}[{\em(a)}]
		\item The following statements are equivalent:
		\begin{enumerate}[{\em(i)}]
			\item $\FRec(T) \neq \varnothing$;
			
			\item $\UFRec(T) \neq \varnothing$;
			
			\item $\RRec(T) \neq \varnothing$;
			
			\item $\RRec(T_i)\neq \varnothing$ for every $1\leq i\leq N$.
		\end{enumerate}
		
		\item The following statements are equivalent:
		\begin{enumerate}[{\em(i)}]
			\item $T$ is frequently recurrent;
			
			\item $T$ is $\Uc$-frequently recurrent;
			
			\item $T$ is reiteratively recurrent;
			
			\item $T_i$ is reiteratively recurrent for every $1\leq i\leq N$.
		\end{enumerate}
	\end{enumerate}
\end{theorem}
\begin{proof}
	We clearly have the inclusion
	\[
	\FRec(T) \subset \prod_{i=1}^N \RRec(T_i).
	\]
	Now given $\symbf{x}_0 = (x_1,...,x_N) \in X$ such that $x_i \in \RRec(T_i)$ for each $1\leq i\leq N$, let us show that $\symbf{x}_0 \in \cl{\FRec(T)}^{\tau_X}$. Applying Theorem \ref{The:Main} we obtain a $T_i$-invariant measure $\mu_{x_i}$ on $X_i$ such that $x_i \in \supp(\mu_{x_i})$ for each $1\leq i\leq N$. Since
	\[
	\Bi(X,\tau_X) = \prod_{i=1}^N \Bi(X,\tau_{X_i}),
	\]
	we can consider the product measure $\mu_{\symbf{x}_0} := \prod_{i=1}^N \mu_{x_i}$ on the product space $X$, which is a $T$-invariant measure (see \cite[Theorem 1.1 and Definition 1.2]{Walters1982}) for which $\symbf{x}_0 \in \supp(\mu_{\symbf{x}_0})$. Applying now Lemma \ref{Lem:Sophie} we deduce that $\symbf{x}_0 \in \cl{\FRec(T)}^{\tau_X}$.
\end{proof}

The following immediate corollaries yield a product version of Theorem \ref{The:Banach}:

\begin{corollary}
	Let $N\in\NN$ and consider for each $1\leq i\leq N$ an adjoint operator $T_i:~X_i\rightarrow~X_i$ on a separable dual Banach space $X_i$. Then, for the direct sum operator $T=T_1\oplus \cdots\oplus T_N : X \rightarrow X$ on the direct sum space $X = X_1\oplus\cdots\oplus X_N$, we have the equality
	\[
	\cl{\FRec(T)} = \prod_{i=1}^N \cl{\RRec(T_i)}.
	\]
	In particular, the following statements are equivalent:
	\begin{enumerate}[{\em(i)}]
		\item $T$ is frequently recurrent;
		
		\item $T$ is $\Uc$-frequently recurrent;
		
		\item $T$ is reiteratively recurrent;
		
		\item $T_i$ is reiteratively recurrent for every $1\leq i\leq N$.
	\end{enumerate}
	Moreover, the result holds whenever some of the $T_i$ are operators defined on some reflexive Banach spaces $X_i$.
\end{corollary}

In the statement above, and whenever we consider a direct sum space $X_1\oplus\cdots\oplus X_N$, one can use any norm defining the usual product topology on $X_1\oplus\cdots\oplus X_N$ (see Theorem~\ref{The:productEIG}).

\begin{definition}
	Let $(X,T)$ be a linear dynamical system and let $n\in\NN$. We will denote by $T_n:X^n\rightarrow X^n$ the {\em $n$-fold direct sum} of $T$ with itself, i.e. the dynamical system
	\[
	T_n := \underbrace{T\oplus\cdots\oplus T}_{n} : \underbrace{X\oplus\cdots\oplus X}_{n} \longrightarrow \underbrace{X\oplus\cdots\oplus X}_{n},
	\]
	where $X^n:=\underbrace{X\oplus\cdots\oplus X}_{n}$ is the {\em $n$-fold direct sum} of $X$ with itself.
\end{definition}

\begin{corollary}
	Let $T:X\rightarrow X$ be an adjoint operator on a separable dual Banach space $X$. Then the following statements are equivalent:
	\begin{enumerate}[{\em(i)}]
		\item for every $n\in\NN$, $T_n$ is frequently recurrent;
		
		\item for every $n\in\NN$, $T_n$ is $\Uc$-frequently recurrent;
		
		\item for every $n\in\NN$, $T_n$ is reiteratively recurrent;
		
		\item $T$ is reiteratively recurrent.
	\end{enumerate}
	In particular, the result holds whenever $T$ is an operator on a reflexive Banach space $X$.
\end{corollary}

As a consequence of the above fact we can prove some results related with hypercyclicity. We start with an independent proof of \cite[Theorem 2.5 and Corollary 2.6]{ErEsMe2021} for the particular case of the reiteratively hypercyclic (adjoint) operators:

\begin{theorem}
	Let $T:X\rightarrow X$ be a reiteratively hypercyclic adjoint operator on a separable dual Banach space $X$. Then for every $n \in \NN$ the operator $T_n$ is reiteratively hypercyclic and frequently recurrent. In particular, the result holds whenever $T$ is an operator on a separable reflexive Banach space $X$.
\end{theorem}
\begin{proof}
	Let $n \in \NN$. Since $T$ is reiteratively hypercyclic we know that:
	\begin{enumerate}[(a)]
		\item $T$ is topologically weakly mixing (see \cite[Page 548]{BMPP2016}), and hence $T_n$ is topologically transitive, and in particular hypercyclic;
		
		\item $T$ is reiteratively recurrent, and by the above results $T_n$ is frequently recurrent, in particular reiteratively recurrent.
	\end{enumerate}
	By \cite[Theorem 2.1]{BoGrLoPe2020}, reiterative recurrence plus hypercyclicity imply reiterative hypercyclicity. We deduce that $T_n$ is reiteratively hypercyclic and frequently recurrent.
\end{proof}

If we start just with reiterative recurrence, having a dense set of orbits converging to $0$ implies a strong notion of hypercyclicity:

\begin{theorem}
	Let $T:X\rightarrow X$ be an adjoint operator on a separable dual Banach space $X$. Suppose that there is a dense set $X_0 \subset X$ such that $T^kx\to 0$ as $k\to\infty$ for each $x \in X_0$. The following statements are equivalent:
	\begin{enumerate}[{\em(i)}]
		\item for every $n \in \NN$, $T_n$ is $\Uc$-frequently hypercyclic and frequently recurrent;
		
		\item $T$ is reiteratively recurrent.
	\end{enumerate}
	In particular, the result holds if $T$ is an operator on a separable reflexive Banach space $X$.
\end{theorem}
\begin{proof}
	Clearly (i) implies (ii) even if $T:X\rightarrow X$ is not a linear map. If we suppose (ii) and we fix $n \in \NN$, by the above results we get that $T_n$ is frequently recurrent and in particular $\Uc$-frequently recurrent. Let $Y_0 := X_0\oplus\cdots\oplus X_0$ the $n$-direct sum of the set $X_0$. Then $Y_0$ is a dense subset of the $n$-fold direct sum $X^n$ and every orbit of a point of $Y_0$ converges to $(0,...,0) \in X^n$. By \cite[Theorem 2.12]{BoGrLoPe2020}, the existence of $Y_0$ and the $\Uc$-frequent recurrence imply that $T_n$ is $\Uc$-frequently hypercyclic.
\end{proof}

It would be interesting to change the assumption of {\em $\Uc$-frequent hypercyclicity} in the above statement into that of {\em frequent hypercyclicity}. However, as exposed in \cite[Question~2.13]{BoGrLoPe2020}, the following is an open problem:

\begin{question}[\textbf{\cite[Question 2.13]{BoGrLoPe2020}}]
	Let $T$ be a frequently recurrent operator admitting a dense set of vectors with orbit convergent to $0$. Is $T$ is frequently hypercyclic?
\end{question}

If now we focus on Theorems \ref{The:Hilbert} and \ref{The:JaDeGli+URec}, their generalizations for product linear dynamical systems follow in a much easier way, since any $N$-tuple formed by unimodular eigenvectors is a linear combination of such vectors for the direct sum map:

\begin{theorem}\label{The:productEIG}
	Let $N\in\NN$ and suppose that for each $1\leq i\leq N$:
	\begin{enumerate}[{\em(a)}]
		\item we have an operator $T_i:H_i\rightarrow H_i$ on a complex Hilbert space $H_i$. Then, for the direct sum operator $T=T_1\oplus \cdots\oplus T_N : H \rightarrow H$ on the direct sum Hilbert space $H = H_1\oplus\cdots\oplus H_N$, we have the equality:
		\[
		\cl{\lspan(\Ec(T))} = \displaystyle \prod_{i=1}^N \cl{\RRec^{bo}(T_i)},
		\]\newpage
		In particular, the following statements are equivalent:
		\begin{enumerate}[{\em(i)}]
			\item the set $\lspan(\Ec(T))$ is dense in $H$;
			
			
			
			
			
			
			
			\item the set $\RRec^{bo}(T_i)$ is dense in $H_i$ for every $1\leq i\leq N$.
		\end{enumerate}
		
		\item we have a power-bounded operator $T_i:X_i\rightarrow X_i$ on a complex reflexive Banach space $X_i$. Then, for the direct sum operator $T=T_1\oplus \cdots\oplus T_N : X \rightarrow X$ on the direct sum space $X = X_1\oplus\cdots\oplus X_N$, we have the equality
		\[
		\cl{\lspan(\Ec(T))} = \prod_{i=1}^N \cl{\URec(T_i)}.
		\]
		In particular, the following statements are equivalent:
		\begin{enumerate}[{\em(i)}]
			\item the set $\lspan(\Ec(T))$ is dense in $X$;
			
			
			
			
			\item the set $ \URec(T_i)$ is dense in $X_i$ for every $1\leq i\leq N$.
		\end{enumerate}
	\end{enumerate} 
\end{theorem}
\begin{proof}
	Since the vector $(0,...,0,x_i,0,...,0)$ belongs to $\Ec(T)$ whenever $x_i \in \Ec(T_i)$ for each $1\leq i\leq N$, it is enough to apply Theorems \ref{The:Hilbert} and \ref{The:JaDeGli+URec} to each operator $T_i$.
\end{proof}

Finally we get the desired generalization of Theorems \ref{The:Hilbert} and \ref{The:JaDeGli+URec}:

\begin{corollary}
	Let $T \in \Lc(H)$ where $H$ is a complex Hilbert space. The following statements are equivalent:
	\begin{enumerate}[{\em(i)}]
		\item for every $n \in \NN$, the set $\lspan(\Ec(T_n))$ is dense in $H^n$;
		
		\item for every $n \in \NN$, $T_n$ is $\Del^*$-recurrent;
		
		\item for every $n \in \NN$, $T_n$ is $\IP^*$-recurrent;
		
		\item for every $n \in \NN$, $T_n$ is uniformly recurrent;
		
		\item for every $n \in \NN$, the set $\FRec^{bo}(T_n)$ is dense in $H^n$;
		
		\item for every $n \in \NN$, the set $\UFRec^{bo}(T_n)$ is dense in $H^n$;
		
		\item for every $n \in \NN$, the set $\RRec^{bo}(T_n)$ is dense in $H^n$;
		
		\item the set $\RRec^{bo}(T)$ is dense in $H$.
	\end{enumerate}
\end{corollary}

\begin{corollary}
	Let $T:X\rightarrow X$ be a power-bounded operator on a complex reflexive Banach space $X$. The following statements are equivalent:
	\begin{enumerate}[{\em(i)}]
		\item for every $n \in \NN$, the set $\lspan(\Ec(T_n))$ is dense in $X^n$;
		
		\item for every $n \in \NN$, $T_n$ is $\Del^*$-recurrent;
		
		\item for every $n \in \NN$, $T_n$ is $\IP^*$-recurrent;
		
		\item for every $n \in \NN$, $T_n$ is uniformly recurrent;
		
		\item $T$ is uniformly recurrent.
	\end{enumerate}
\end{corollary}

\section{Inverse dynamical systems}\label{Sec:6Inverse}

As in the case of products, given a (topological) dynamical system $T:X\rightarrow X$ with some property, it is natural to ask whether the {\em inverse dynamical system} $T^{-1}:X\rightarrow X$ (if it exists and is continuous) has the same property. This is true for hypercyclicity and reiterative hypercyclicity (see \cite{BoGr2018}), but it fails for $\Uc$-frequent hypercyclicity (see \cite{Menet2019U-freq}) and frequent hypercyclicity (see \cite{Menet2019freq}). It is also known that the inverse of a frequently hypercyclic operator is $\Uc$-frequently hypercyclic (see \cite[Proposition 20]{BaRu2015}).\\[-5pt]

If we focus on recurrence properties, the inverse of a recurrent operator is again recurrent as \cite[Proposition 2.6]{CoMaPa2014} shows. A simpler proof (in a transitive style) of that fact would be:

\begin{proposition}[\textbf{\cite[Proposition 2.6]{CoMaPa2014}}]
	Let $T:X\rightarrow X$ be an invertible operator. Then $T$ is recurrent if and only if so is $T^{-1}$.
\end{proposition}
\begin{proof}
	By \cite[Proposition 2.1]{CoMaPa2014} the result follows from the equivalence
	\[
	T^n(U)\cap U\neq \varnothing \text{ if and only if } U \cap T^{-n}(U) \neq \varnothing,
	\]
	valid for any non-empty open subset $U$ of $X$.
\end{proof}

However, it is also shown in \cite[Remark 2.7]{CoMaPa2014} that the sets $\Rec(T)$ and $\Rec(T^{-1})$ may not be equal in spite of the fact that their closures coincide. For general $\Fc$-recurrence notions the following problem was proposed in \cite{BoGrLoPe2020}:

\begin{question}[\textbf{\cite[Question 2.14]{BoGrLoPe2020}}]\label{Q:invertible}
	Let $T$ be an invertible operator. If $T$ is reiteratively (resp. $\Uc$-frequently, frequently or uniformly) recurrent, does $T^{-1}$ has the same property?
\end{question}

In fact, we could ask the same question for $\IP^*$, $\Del^*$-recurrence and unimodular eigenvalues. However, for the latest the linearity is enough to show it since if
\[
Tx = \lambda x \text{ for some } \lambda \in \TT \text{, then } T^{-1}x = \cl{\lambda} (T^{-1}\lambda x) = \cl{\lambda} x,
\]
so clearly $\lspan(\Ec(T)) = \lspan(\Ec(T^{-1}))$. In order to answer Question \ref{Q:invertible} in our dual/reflexive or Hilbertian setting we just have to recall the following trivial fact: given a homeomorphism $T:X\rightarrow X$ on a Polish space $X$ and a Borel measure $\mu$ on $X$, $\mu$ is $T$-invariant if and only if it is $T^{-1}$-invariant.

\begin{theorem}[\textbf{From Reiterative to Inverse Frequent Recurrence}]
	Let $T:X\rightarrow X$ be a homeomorphism of the Polish space $(X,\tau_X)$, and assume that $X$ is endowed with a Hausdorff topology $\tau_1$ which fulfills {\em(I)} for $T$, {\em(II)}, and {\em(III*)}. Then we have
	\[
	\cl{\FRec(T)}^{\tau_X} = \cl{\RRec(T)}^{\tau_X} \subset \cl{\FRec(T^{-1})}^{\tau_X} \subset \cl{\RRec(T^{-1})}^{\tau_X}.
	\]
	Moreover:
	\begin{enumerate}[{\em(a)}]
		\item If $T$ is reiteratively recurrent then $T^{-1}$ is frequently recurrent.
		
		\item If $\RRec(T)\neq\varnothing$ then $\FRec(T)\cap\FRec(T^{-1})\neq\varnothing$.
		
		\item If $X$ can be endowed with a Hausdorff topology $\tau_2$ which fulfills {\em(I)} for $T^{-1}$, {\em(II)}, and {\em(III*)}, then the above inclusions are equalities and $T$ is reiteratively (and hence frequently) recurrent if and only if so is $T^{-1}$.
	\end{enumerate}
\end{theorem}
\begin{proof}
	The equality is shown in Theorem \ref{The:Polish}. The first inclusion follows from Lemma~\ref{Lem:Sophie} applied to the measures constructed with Theorem \ref{The:Main} for each point of $\RRec(T)$, using the fact that they are $T^{-1}$-invariant. The second inclusion follows by definition. Moreover, if there exists $x_0 \in \RRec(T)$ and if we take the invariant probability measure $\mu_{x_0}$ on $X$ constructed with Theorem~\ref{The:Main}, then by Lemma \ref{Lem:Sophie} we have
	\[
	\mu_{x_0}(\FRec(T)) = 1 = \mu_{x_0}(\FRec(T^{-1}))
	\]
	which implies that $\FRec(T)\cap\FRec(T^{-1})\neq\varnothing$. Finally, if such a topology $\tau_2$ exists we can apply the first part of the result to $T^{-1}$ obtaining $\cl{\RRec(T^{-1})}^{\tau_X} \subset \cl{\FRec(T)}^{\tau_X}$.
\end{proof}

As a corollary of the above theorem, and using the arguments from Theorem \ref{The:Banach} we have:

\begin{corollary}
	Let $T:X\rightarrow X$ be an invertible adjoint operator on a separable dual Banach space $X$. Then we have the equalities
	\[
	\cl{\RRec(T)} = \cl{\FRec(T)} = \cl{\FRec(T^{-1})} = \cl{\RRec(T^{-1})}.
	\]
	Moreover:
	\begin{enumerate}[{\em(a)}]
		\item $T$ is reiteratively (and hence frequently) recurrent if and only if so is $T^{-1}$.
		
		\item If $\RRec(T)\setminus\{0\}\neq\varnothing$ then $[\FRec(T)\cap\FRec(T^{-1})]\setminus\{0\}\neq\varnothing$.
	\end{enumerate}
	In particular, the result holds whenever $T$ is an operator on a reflexive Banach space $X$.
\end{corollary}
\begin{proof}
	Let $S:Y\rightarrow Y$ be an operator on a Banach space $Y$ such that $Y^*=X$ and $S^*=T$. It is a known fact that $T$ is invertible if and only if $S$ is invertible, and in this case, $T^{-1} = (S^{-1})^*$, so $T^{-1}$ is also an adjoint operator on the separable Banach space $X$ and hence it is $w^*$-continuous. The result follows from the above theorem applied to $T:X\rightarrow X$, $(X,\|\cdot\|)$ and the topology $w^*$.
\end{proof}

With the above fact we give an alternative prove of \cite[Theorem 3.6]{BoGr2018} for adjoint operators:

\begin{theorem}\label{The:inverseRHC}
	Let $T:X\rightarrow X$ be an invertible adjoint operator on a separable dual Banach space. If $T$ is reiteratively hypercyclic (and hence frequently recurrent) then so is $T^{-1}$. In particular, the result holds whenever $T$ is an operator on a separable reflexive Banach space $X$.
\end{theorem}
\begin{proof}
	By the above theorem $T^{-1}$ is frequently recurrent and in particular reiteratively recurrent. Since hypercyclicity (or transitivity) is also preserved by taking the inverse system, \cite[Theorem 2.1]{BoGrLoPe2020} implies that $T^{-1}$ is reiteratively hypercyclic.
\end{proof}

We cannot change the assumption of {\em reiterative hypercyclicity} in the statement of Theorem~\ref{The:inverseRHC} above into the assumption of {\em $\Uc$-frequent hypercyclicity} since there are invertible $\Uc$-frequently hypercyclic operators on $\ell^p(\NN)$ ($1\leq p<\infty$) whose inverse is not $\Uc$-frequently hypercyclic (see \cite{Menet2019U-freq}). However, it would be interesting to know whether it is possible to change the assumption of {\em reiterative hypercyclicity} into that of {\em frequent hypercyclicity}: even though it is known that there are invertible frequently hypercyclic operators on $\ell^1(\NN)$ whose inverse is not frequently hypercyclic (see \cite{Menet2019freq}), one can check that these are not adjoint operators and moreover by \cite[Proposition 20]{BaRu2015} the inverse of a frequently hypercyclic operator is always $\Uc$-frequently hypercyclic.\\[-5pt]

All the counterexamples mentioned here are C-type operators, which were introduced for the first time in \cite{Menet2017} and further developed in \cite{GriMaMe2021,Menet2019U-freq,Menet2019freq}, so a possible counterexample for the frequent hypercyclicity case could arise from those operators. If, on the other hand, one wishes to prove an analogue of Theorem~\ref{The:inverseRHC} for the frequent hypercyclicity case in our dual/reflexive framework, one cannot take a similar approach since there are chaotic operators, which are in particular frequently recurrent and hypercyclic, but not $\Uc$-frequently hypercyclic (see \cite{Menet2017} and \cite{GriMaMe2021}) and hence not frequently hypercyclic.\\[-5pt]

If we now focus on uniform, $\IP^*$ and $\Del^*$-recurrence, Theorems \ref{The:Hilbert} and \ref{The:JaDeGli+URec} combined with the equality $\lspan(\Ec(T)) = \lspan(\Ec(T^{-1}))$ give us the following:

\begin{corollary}
	Let $T:H\rightarrow H$ be an invertible operator on a complex Hilbert space $H$. Then we have the equalities
	\[
	\cl{\RRec^{bo}(T)} = \cl{\lspan(\Ec(T))} = \cl{\lspan(\Ec(T^{-1}))} = \cl{\RRec^{bo}(T^{-1})}.
	\]
	In particular, $T$ is uniformly (and hence $\IP^*$ and $\Del^*$) recurrent if and only if so is $T^{-1}$.
\end{corollary}

\begin{corollary}
	Let $T:X\rightarrow X$ be an invertible operator on a complex reflexive space $X$. If $T$ is power-bounded, then we have
	\[
	\cl{\URec(T)} = \cl{\lspan(\Ec(T))} = \cl{\lspan(\Ec(T^{-1}))} \subset \cl{\URec(T^{-1})}.
	\]
	In particular, if $T$ is uniformly recurrent then $\lspan(\Ec(T^{-1}))$ is a dense set in $X$. Moreover, if $T^{-1}$ is also power-bounded then the above inclusion is an equality and $T$ is uniformly (and hence $\IP^*$ and $\Del^*$) recurrent if and only if so is $T^{-1}$.
\end{corollary}

\section{How typical is a reiteratively recurrent operator?}\label{Sec:7Typical}

Let $H$ be a complex separable Hilbert space. For any $M>0$, denote by $\Lc_M(H)$ the set of bounded operators $T \in \Lc(H)$ such that $\|T\|\leq M$. Our aim in this short section is to present a result pertaining to the typicality of reiteratively recurrent operators of $\Lc_M(H)$, with $M>1$, for one of the two (Polish) topologies SOT and SOT$^*$. The framework that we use here is presented in detail in \cite[Chapters 2 and 3]{GriMaMe2021}, so we will be rather brief in our presentation and refer the readers to the works \cite{GriMa2021}, \cite{GriMaMe2021} or \cite{GriMaMe2021b} for more on typical properties of operators on Hilbert or Banach spaces.\\[-5pt]

We recall that the {\em Strong Operator Topology} (SOT) on $\Lc(H)$ is defined as follows: any $T_0 \in \Lc(H)$ has a SOT-neighbourhood basis consisting of sets of the form
\[
U_{T_0,x_1,...,x_s,\eps} := \{ T \in \Lc(H) : \|(T-T_0)x_i\|<\eps \text{ for } i=1,...,s\},
\]
where $x_1,...,x_s \in H$ and $\eps>0$.\\[-5pt]

The {\em Strong$^*$ Operator Topology} (SOT$^*$) is the ``self-adjoint'' version of SOT: a basis of SOT$^*$-neighbourhoods of $T_0 \in \Lc(H)$ is provided by the sets of the form
\[
V_{T_0,x_1,...,x_s,\eps} := \{ T \in \Lc(H) : \|(T-T_0)x_i\|<\eps \text{ and } \|(T-T_0)^*x_i\|<\eps \text{ for } i=1,...,s\},
\]
where $x_1,...,x_s \in H$ and $\eps>0$.\\[-5pt]

It is easily shown that $(\Lc_M(H),\text{SOT})$ and $(\Lc_M(H),\text{SOT}^*)$ are Polish spaces for any $M>0$ (see \cite[Section 4.6.2]{Pedersen1989}), and hence, a property of operators $T \in \Lc_M(H)$ will be called {\em typical} if the set of operators fulfilling it is co-meager (i.e. contains a dense $G_{\delta}$ set), and {\em atypical} if its negation is typical. Following the notation used in \cite{GriMaMe2021} we can write
\begin{eqnarray}
	\HC(H) &:=& \{ T \in \Lc(H) \text{ hypercyclic} \}; \nonumber\\[5pt]
	\text{INV}(H) &:=& \{ T \in \Lc(H) \text{ admitting a non-trivial invariant measure} \}; \nonumber
\end{eqnarray}
and for each $M>1$ the set $\HC_M(H)$ is defined as $\HC(H) \cap \Lc_M(H)$. Following the spirit of this study we introduce the following notation:
\begin{eqnarray}
	\RHC(H) &:=& \{ T \in \Lc(H) \text{ reiteratively hypercyclic} \}; \nonumber\\[5pt]
	\RRec(H) &:=& \{ T \in \Lc(H) \text{ reiteratively recurrent} \}; \nonumber\\[5pt]
	\RRec^{\neq\varnothing}(H) &:=& \{ T \in \Lc(H) : \RRec(T)\setminus\{0\}\neq\varnothing \}; \nonumber
\end{eqnarray}
and as it was done previously for the set $\HC(H)$, for each $M>1$, we will denote the respective bounded versions of these sets of operators by $\RHC_M(H)$, $\RRec_M(H)$ and $\RRec_M^{\neq\varnothing}(H)$.\\[-5pt]

Let us first recall that a SOT-typical operator in $\Lc_M(H)$, for $M>1$, has any form of recurrence one can wish for: by \cite{EiMat2013} it is known that a typical $T \in \Lc_M(H)$ is unitarily similar to $MB_{\infty}$, where $B_{\infty}$ denotes the backward shift of infinite multiplicity on $\ell^2(\NN,\ell^2(\NN))$, and $MB_{\infty}$ is such that the linear span of its unimodular eigenvectors is dense in $\ell^2(\NN,\ell^2(\NN))$, see \cite[Theorem~5.2]{EiMat2013}.\\[-5pt]

With respect to the topology SOT$^*$, it is proved in \cite[Theorem 2.29]{GriMaMe2021} that for every $M>1$ the set $\HC_M(H)\setminus\textup{INV}(H)$ is co-meager in the space $(\HC_M(H),\textup{SOT}^*)$. In other words, a SOT$^*$-typical hypercyclic operator on $H$ admits no non-trivial invariant measure. Combining Theorem \ref{The:Main} of the present work with \cite[Theorem 2.29]{GriMaMe2021} we obtain:

\begin{corollary}\label{Cor:SOT*}
	For every $M>1$, the set $\RRec^{\neq\varnothing}_M(H)$ is meager in $(\Lc_M(H),\textup{SOT}^*)$. In other words, a \textup{SOT}$^*$-typical operator on $H$ has no non-zero reiteratively recurrent point.
\end{corollary}

Moreover, since \cite[Theorem 2.1]{BoGrLoPe2020} implies that reiterative recurrence plus hypercyclicity is equivalent to reiterative hypercyclicity, we have that $\RHC(H) = \RRec(H)\cap\HC(H)$. Using Corollary \ref{Cor:SOT*} we can improve \cite[Corollary 2.36]{GriMaMe2021} in terms of reiterative recurrence:

\begin{corollary}
	For every $M>1$, the set $\RRec^{\neq\varnothing}_M(H) \cap \HC_M(H)$ is meager in the space $(\HC_M(H),\textup{SOT}^*)$. In particular, the set $\RHC_M(H)$ is meager in $(\HC_M(H),\textup{SOT}^*)$. In other words, a \textup{SOT}$^*$-typical hypercyclic operator on $H$ does not admit any non-zero reiteratively recurrent point, and, in particular, is not reiteratively hypercyclic.
\end{corollary}

\section{Open problems}\label{Sec:8Open}

In this section we gather some possibly interesting open questions and a few comments related to them. We start by Questions \ref{Q:urecIP*rec} and \ref{Q:urecEspan}, already stated in Subsection \ref{Subsec:1.3UIP*D*UEig}, which we recall here with some extra generality:

\begin{question}[\textbf{Question \ref{Q:urecEspan}}]\label{Q:urecFRECHET}
	Let $T$ be a uniformly recurrent operator on a complex Fr\'echet space $X$. Is $\lspan(\Ec(T))$ a dense set in $X$? What about the cases where $T$ is an adjoint operator on a separable dual Banach space or where $X$ is a reflexive Banach space?
\end{question}

\begin{question}[\textbf{From \cite[Question 6.3]{BoGrLoPe2020} and Question \ref{Q:urecIP*rec}}]\label{Q:Del*.IP*.unif.FRECHET}
	Does there exist an operator (possibly on a Fr\'echet space) which is uniformly recurrent but not $\Del^*$-recurrent? What about distinguishing uniform recurrence from $\IP^*$-recurrence?
\end{question}


Note that these two questions make sense in the more general context of {\em complex Fr\'echet spaces}, and in fact both questions are still unsolved for that rather general class of spaces. It is clear that, in every possible {\em complex} context, a positive answer to Question~\ref{Q:urecFRECHET} implies a negative one to Question~\ref{Q:Del*.IP*.unif.FRECHET}. Moreover, it would even imply a negative answer for the {\em real} case of Question~\ref{Q:Del*.IP*.unif.FRECHET}: given any uniformly recurrent {\em real} linear dynamical system we could consider its complexification, and by the product-arguments used for Theorem~\ref{The:productEIG} we would get unimodular eigenvectors and hence $\Del^*$-recurrence; the initial {\em real} dynamical system could possibly not contain the obtained unimodular eigenvectors, but the real and complex parts of such vectors would clearly be $\Del^*$-recurrent for the original {\em real} system. It is worth mentioning that uniform and $\IP^*$-recurrence are completely distinguished for compact dynamical systems (see the construction from \cite{FaLi1998}, its properties in \cite{ChLiWa2006} and use \cite[Theorems 1.15 and 9.12]{Furstenberg1981}), so the question is if the linearity avoids that distinction.\\[-5pt]


The technique used in the proof of Theorem \ref{The:Hilbert} (via Gaussian measures) is very different from the one used in Theorem \ref{The:JaDeGli+URec} (via the Jacobs-Deleeuw-Glicksberg theorem). Indeed we loose the contact with measures and the unimodular eigenvectors are obtained from a totally different construction (see \cite[Section 2.4]{Krengel1985}). It seems to us that a more general ``eigenvectors' constructing machine'', not restricted to the measures or power-bounded assumptions, should be developed in order to provide a better answer to Question \ref{Q:urecFRECHET}. What we know for the moment, leaving apart the power-bounded case which seems very specific, is the following:

\begin{proposition}
	Let $T \in \Lc(H)$, where $H$ is a complex separable Hilbert space. Given a $T$-invariant $w$-compact subset $K$ of $H$ for which $0 \notin K$, we have
	\[
	\cl{\lspan(\Ec(T))} \cap K \neq \varnothing \quad \text{ and in particular } \quad \Ec(T)\neq\varnothing.
	\]
\end{proposition}
\begin{proof}
	We have that $T\res_K:(K,w)\rightarrow (K,w)$ is a $w$-compact dynamical system, so it admits a $T\res_K$-invariant probability measure $\mu$ on $K$ (see \cite[Page 62]{Furstenberg1981}). Since the norm topology and the weak topology on $H$ have the same Borel sets, we can extend the measure $\mu$ into a Borel probability measure on the whole space $H$ (still denoted by $\mu$) using the formula
	\[
	\mu(A) := \mu(K\cap A) \quad \text{ for every Borel set } A \in \Bi(H).
	\]
	Note that $\mu$ is $T$-invariant.
	We deduce that:
	\begin{enumerate}[(a)]
		\item $\mu$ is non-trivial, since $0 \notin K$;
		
		\item $\mu$ has a finite second-order moment, since $\supp(\mu) \subset K$.
	\end{enumerate}
	Lemma \ref{Lem:Eigenvectors} implies that $\cl{\lspan(\Ec(T))}\cap K\neq\varnothing$ and in particular $\Ec(T)\neq\varnothing$.
\end{proof}

\begin{remark}
	Let $H$ be a complex separable Hilbert space. Since the set $\cl{\orb(x,T)}^w$ is a $T$-invariant $w$-compact subset of $H$ for any point $x \in H$ with bounded $T$-orbit, the arguments of the above proposition imply that {\em for any $M>1$, a \textup{SOT}$^*$-typical operator $T \in \Lc_M(H)$ has the property that every bounded orbit of $T$ contains $0$ in its weak closure}.
\end{remark}

\begin{proposition}
	Let $T$ be an adjoint operator on a complex separable dual Banach space $X$. Let $n \in \NN$ and $\lambda \in \TT$. Given a $[\lambda T]^n$-invariant $w^*$-compact and convex subset $K$ of $H$ for which $0 \notin K$, we have
	\[
	\Ec(T) \cap \lspan(\orb(x,T)) \neq \varnothing \quad \text{ for some } x \in K,
	\]
	and in particular $\Ec(T)\neq\varnothing$.
\end{proposition}
\begin{proof}
	By the Schauder's Fixed-Point theorem there is $x \in K$ for which the identity $[\lambda T]^nx=x$ holds. Taking $\alpha=\lambda^{-n} \in \TT$ we get that $(\alpha-T^n)x=0$. If we split the polynomial $(\alpha-z^n) \in \CC[z]$ we have
	\[
	(\alpha-z^n) = \prod_{i=1}^n (\alpha_i-z),
	\]
	where the $\alpha_i$'s are distinct $n$-th roots of $\alpha$ in $\TT$. Considering the vectors
	\[
	y_0:=x \quad \text{ and } \quad y_j := (\alpha_j-T)y_{j-1} = \prod_{i=1}^j (\alpha_i-T)x \quad \text{ for each } 1\leq j\leq n,
	\]
	we have $y_0\neq 0$ since $0 \notin K$, but $y_n=(\alpha-T^n)x=0$. Then for some $0\leq k\leq n-1$ we have that $y_k \in \Ec(T) \cap \lspan(\orb(x,T))$. In particular $\Ec(T)\neq\varnothing$.
\end{proof}


Another natural question concerning Theorem \ref{The:Hilbert} is the relevance, in assertions (v) to (vii) of both parts (a) and (b), of the assumption that the vectors under consideration have {\em bounded} orbit. This fact is used in order to ensure that the invariant measures, which by Theorem \ref{The:Main} can be constructed from each reiteratively recurrent vector, have a finite second-order moment. To omit this boundedness assumption (or weak versions of it) seems to require new ideas. We recall here the following open problem from \cite{GriMaMe2021}:

\begin{question}[\textbf{\cite[Question 8.3]{GriMaMe2021}}]
	Does there exist an operator on a complex separable Hilbert space admitting a non-trivial invariant probability measure but no eigenvalues?
\end{question}


The following product and inverse questions also remain open:

\begin{question}\label{Q:product}
	Let $T$ be an operator on a Fr\'echet space $X$. If $T$ is reiteratively (resp. $\Uc$-frequently, frequently, uniformly) recurrent, does the $n$-fold direct sum $T\oplus\cdots\oplus T$, $n\geq 2$, have the same property?
\end{question}

\begin{question}[\textbf{From \cite[Question 2.14]{BoGrLoPe2020} and Question \ref{Q:invertible}}]
	Let $T$ be an invertible operator on a Fr\'echet space $X$. If $T$ is reiteratively (resp. $\Uc$-frequently, frequently, uniformly, $\IP^*$ or $\Del^*$) recurrent, does $T^{-1}$ have the same property?
\end{question}

It is worth mentioning that Question \ref{Q:product} is also open for usual recurrence, as defined in Subsection \ref{Subsec:1.1GB}, and it seems to be a non-trivial question. For the $\IP^*$ and $\Del^*$ cases, the fact that such families are filters (see \cite{BerDown2008}) implies a positive answer.\\[-5pt]


As mentioned in \cite{BoGrLoPe2020}, the set of periodic points $\Per(T)$ of an operator $T$ has the property that $\Per(T)$ is either equal to $X$ or is a meager set (by the Baire Category theorem, either $\Per(T)$ is of first category or else $T^n = I$ for some $n \in \NN$). The same phenomenon happens (at least when $X$ is a Banach space) with the set of uniformly recurrent vectors $\URec(T)$, since by \cite[Corollary~3.2]{BoGrLoPe2020} if $\URec(T)$ is co-meager in $X$ then $T$ is power-bounded and $\URec(T)=X$. This motivates the following question:

\begin{question}[\textbf{\cite[Question 2.9]{BoGrLoPe2020}}]\label{Q:categoryFRec}
	Let $T$ be an operator on a Fr\'echet space $X$. Do we always have that either $\FRec(T)=X$ or $\FRec(T)$ is a meager set?
\end{question}

This seems to be a non-trivial question even in the dual/reflexive setting, since the frequently recurrent points obtained in our construction form a ``big'' set with respect to a certain invariant measure, and usually this has nothing to do with the ``bigness'' from the topological point of view (i.e. in the Baire Category sense). In fact, given any chaotic operator $T:X\rightarrow X$ (i.e. hypercyclic with a dense set of periodic points), it admits an invariant probability measure $\mu$ on $X$ with full support (see \cite[Corollary 3.6]{Grosse2006}) and hence $\mu(\FRec(T))=1$ by Lemma \ref{Lem:Sophie}. However, since $T$ is hypercyclic we have that $\FRec(T)$ is a meager set, otherwise by \cite[Theorem 2.7]{BoGrLoPe2020} the set $\FHC(T)=\FRec(T)\cap\HC(T)$ would be co-meager contradicting \cite[Corollary 19]{BaRu2015}.

\noindent
\parbox[t][3cm][l]{9cm}{\small
	\noindent
	Sophie Grivaux\\
	CNRS, Universit\'e de Lille\\
	UMR 8524 - Laboratoire Paul Painlev\'e,\\
	F-59000 Lille, France\\
	E-mail: sophie.grivaux@univ-lille.fr}
\parbox[t][3cm][l]{9cm}{\small
	Antoni L\'opez-Mart\'{\i}nez\\
	Universitat Polit\`ecnica de Val\`encia\\
	Institut Universitari de Matem\`atica Pura i Aplicada,\\
	Edifici 8E, Acces F, 4a planta,	46022 Val\`encia, Spain\\
	E-mail: anlom15a@upv.es}
\end{document}